\documentclass[3p,times]{elsarticle}

\journal{J. Comput. Phys.}

\usepackage{amssymb,amsmath,mathtools}
\usepackage{amsthm}
        \theoremstyle{plain}
        \newtheorem{thm}{Theorem}[section]
        \newtheorem{prop}[thm]{Proposition}
        \newdefinition{rem}{Remark}
\usepackage[section]{algorithm}
\usepackage{algpseudocode}
\usepackage{etoolbox}
\newcommand{\Mat}[1]{#1}
\renewcommand{\vec}[1]{\boldsymbol{#1}}
\renewcommand{\Vec}[1]{\mathbf{#1}}
\newcommand{\blkVec}[1]{\mathbf{#1}}

\newcommand{\tensorOne}[1]{\boldsymbol#1}
\newcommand{\tensorTwo}[1]{\boldsymbol#1}
\newcommand{\funSpace}[1]{\mathcal{#1}}
\newcommand{\vecFunSpace}[1]{\boldsymbol{\mathcal{#1}}}
\newcommand{\tensorFour}[1]{\textbf{\sffamily #1}}

\newcommand{\tS}{\widetilde{S}}
\newcommand{\tH}{\widetilde{H}}
\newcommand{\tT}{\widetilde{T}}

\usepackage{subcaption}
\usepackage{arydshln}
\usepackage{booktabs}
\usepackage{multirow}

\usepackage[nodots]{numcompress}

\usepackage{xcolor}

\usepackage[colorlinks=true, allcolors=blue, breaklinks]{hyperref}
\biboptions{sort&compress}
\usepackage{todonotes}

\begin{document}

\begin{frontmatter}

\title{A scalable preconditioning framework for stabilized contact mechanics with
       hydraulically active fractures}

\author[UNIPD]{Andrea Franceschini\corref{mycorrespondingauthor}}
\ead{andrea.franceschini@unipd.it}

\author[UNIPD]{Laura Gazzola}
\ead{laura.gazzola.1@phd.unipd.it}

\author[UNIPD]{Massimiliano Ferronato}
\ead{massimiliano.ferronato@unipd.it}

\address[UNIPD]{Department of Civil, Environmental and Architectural Engineering,
  University of Padova, Padova, Italy}
\cortext[mycorrespondingauthor]{Corresponding author}

\begin{abstract}
A preconditioning framework for the coupled problem of frictional contact
mechanics and fluid flow in the fracture network is presented. The porous medium is
discretized using low-order continuous finite elements, with cell-centered Lagrange
multipliers and pressure unknowns used to impose the constraints and solve the fluid flow in the fractures, respectively. This formulation does not require any interpolation between
different fields, but is not uniformly inf-sup stable and requires a stabilization. For
the resulting $3 \times 3$ block Jacobian matrix, we design scalable preconditioning
strategies, based on the physically-informed block partitioning of the unknowns and
state-of-the-art multigrid preconditioners. The key idea is to restrict the system to a single-physics problem, approximately solve it by an inner algebraic multigrid approach, and finally prolong it back to the fully-coupled problem. Two different techniques are presented, analyzed and compared by changing the ordering of the
restrictions. Numerical results
illustrate the algorithmic scalability, the impact of the relative number of
fracture-based unknowns, and the performance on a real-world problem.

\end{abstract}


\end{frontmatter}


\allowdisplaybreaks

\section{Introduction}
\label{sec:intro}

In recent years, attention has grown around novel technologies and applications in the
subsurface, like geothermal energy production \cite{pan2019establishment,
wei2019numerical, asai2019efficient}, hydraulic fracturing \cite{williams2019discursive,
tan2019politics, krzaczek2020simulations}, CO$_2$ sequestration \cite{fan2019thermo,
li2019coupled, liu2019tutorial} and underground gas storage \cite{zhou2019seismological,
karev2019geomechanical, firme2019salt}. In these contexts, one of the key components is
the simultaneous simulation of frictional contact mechanics and fluid flow in faults and fractures, which represent tightly coupled
physical processes. 
In fact, the aperture and slippage between the contact surfaces drive the fluid flow in the fractures, while the pressure variation perturbs the stress state in the surrounding medium and influences the contact mechanics itself. 
To achieve the
desired accuracy, large domains are usually required, with high resolution representations
of geological structures and their heterogeneous properties \cite{fergamjantea10,castelletto2013geological}, and, specifically, of faults and
fracture networks \cite{zoback2010reservoir, goodman1968model, ferronato2008numerical,
GarKarTch16,Set_etal17, shakiba2015using, ren2016fully, wong2019investigation,
wu2019integrating, deb2009extended, zhang2011extended, mohammadi2012xfem,
flemisch2016review, Berrone2017768, vahab2017numerical, khoei2018enriched, Berrone2019C317, Berrone2021B381}. It is, therefore, natural to have a growing demand towards the development of sophisticated models of increasing size, which are computationally intensive and require better and better performances. A key factor in this sense is the linear solver, which is usually by far the most
time-consuming component in a real-world simulation \cite{koric2016sparse,
franceschini2019robust}.

In this work, we analyze the simulation of frictional contact
mechanics coupled with the fluid flow in a fracture network and present a scalable and efficient
preconditioning framework for the linear system arising from the discretization and linearization of
the coupled problem. As to the discretization approach, we elect to use the Discrete Fracture Model
(DFM) \cite{GarKarTch16}, i.e., an explicit representation of the fracture surfaces, while
the constraints are imposed with the aid of Lagrange multipliers \cite{hild2010stabilized,
JhaJua14, FraFerJanTea16, berge2020finite, koppel2019stabilized}. As it is common in geological and
reservoir simulations, we rely on low-order finite elements for the mechanics and
a cell-centered finite volume scheme for the fluid flow. Lagrange multipliers are the contact forces acting on
the fracture surfaces as a cell-centered variable, thus sharing the same representation as the fluid
pressure field with no interpolation needed. The details of this
discretization scheme are described in \cite{fr2020alg}. This approach is unstable in the 
Ladyzhenskaya-Babu\v{s}ka-Brezzi
(LBB) sense, i.e., it does not uniformly satisfy the
\textit{inf-sup} condition \cite[Section 3.1]{wohlmuth2011variationally}, and
requires a stabilization. 
In this work, we use the
global algebraic approach introduced in the reference work \cite{fr2020alg}. 
The Jacobian matrix arising
from the described problem is non-symmetric with a $3 \times 3$ block structure, which has to be properly preconditioned to allow for a robust, scalable and efficient solution with the aid of Krylov subspace solvers.

It is well known that iterative methods based on projections/orthogonalizations onto Krylov subspaces \cite{saad2003iterative} are in practice mandatory to solve large and sparse linear systems deriving from the discretization of
PDEs, because they allow for a lower complexity, smaller memory requirement, and better degree of
algorithmic parallelism than  direct methods \cite{davis2006direct}.
However, robustness, scalability and computational efficiency of this class of methods is tightly connected with the choice of a proper preconditioning technique \cite{saad2003iterative}. Roughly speaking, preconditioners are approximate applications of the system matrix inverse, and, from the algebraic viewpoint, can be
classified into three main categories: (i) incomplete factorizations \cite{saad1994ilut,
lin1999incomplete, benzi2002preconditioning}, (ii) approximate inverses
\cite{benzi1996sparse, tang1999toward, huckle2003factorized,janfergam10,janfer11, janna2015fsaipack}, and (iii)
multilevel methods, i.e., domain decomposition \cite{janfergam13,dolean2015introduction,
zampini2016pcbddc, badia2016multilevel, li2017low} and multigrid-like techniques
\cite{mccormick1982multigrid, stuben1983algebraic, brandt1986algebraic, stuben2001review,
notay2012aggregation, brezina2005adaptive, vanvek1996algebraic, brezina2006adaptive,
brandt2011bootstrap, brandt2014bootstrap,
Pasetto20171159, dambra2018bootcmatch, dambra2019improving,
paludetto2019novel}. A key feature for a modern preconditioning framework is the
algorithmic scalability, i.e., the ability to solve an increasingly refined problem with an approximately
constant number of iterations of the Krylov solver. This property is particularly important in view of the development of problems of increasing size by exploiting the availability of massively parallel computational platforms. Incomplete factorizations and approximate inverses can exhibit amazing performances, but do
not have a linear complexity with the system size. By distinction, multilevel methods can have a lower performance on a single system, but are designed to be optimal with respect to the scalability issue. Algebraic multigrid (AMG, \cite{xu2017algebraic}) is one of the most effective multilevel
approaches and consists of the complementary use of: (i) a smoother that reduces high
frequency errors, (ii) a coarse grid correction that reduces low frequency errors, and (iii)
restriction and interpolation operators, to move from one grid to another. Starting from the original
works, e.g., \cite{ruge1987algebraic}, a wide range of multigrid approaches has appeared in the literature,
extending the applicability of this method, originally designed for elliptic PDEs, to both
non-symmetric \cite{manteuffel2018nonsymmetric, manteuffel2019nonsymmetric} and block matrices \cite{webster2016stabilisation, brenner2014multigrid, chen2015multigrid,
brenner2018multigrid, wiesner2021algebraic, brenner2020multigrid}. 
Nonetheless, robustness and efficiency is still an open issue for AMG whenever used 
as a black-box tool in problems with these algebraic properties.
The Jacobian matrix arising from the model considered herein is a
non-symmetric $3 \times 3$ block matrix and, despite the available studies for
similar problems, none of them can be straightforwardly and effectively
applied to our case. In the context of geomechanical simulations, only a few studies
on $2 \times 2$ block Jacobian systems \cite{aagaard2013domain, franceschini2019block,
wiesner2021algebraic} are found by the authors.

The purpose of this work is to design a scalable preconditioning framework for the $3 \times 3$
block matrix arising from the coupled simulation of frictional contact mechanics and fluid flow in the fracture network. The idea is to exploit the inherent physics-based block subdivision and the scalability of AMG techniques available from the literature. 
The full system is first restricted to a single-physics problem, then approximately solved by AMG, and finally prolonged back to the original size.
According to the selected restriction ordering, different approaches can be derived. In this work, we consider two different options and investigate advantages and drawbacks in order to
find the most appropriate algorithm for real-world simulations. 
%
The paper is organized as follows. Section \ref{sec:gen_fram} introduces the
physical problem in both the strong and weak forms, in order to understand the meaning and
features of each block of the Jacobian system. In Section \ref{sec:preconditioner},
the preconditioning framework is presented, with a detailed analysis of two selected options. Finally, Section \ref{sec:numres} presents a set of numerical results with the aim of
comparing the proposed approaches 
and investigating the algorithmic scalability 
in both theoretical and real-world benchmarks.
A few concluding remarks close the paper.

\section{Problem statement}
\label{sec:gen_fram}

We model the deformation of an open elastic domain $\Omega \subset \mathbb{R}^3$, assuming
quasi-static conditions and infinitesimal strains within the open time interval $\mathcal{T}=\left(0,t_{\max}\right]$. We denote by $\partial \Omega$ its boundary, with $\overline{\Omega}=\Omega\cup\partial\Omega$, 
and $\vec{n}_{\Omega}$ the outer normal vector to $\partial \Omega$, while a set of internal boundaries $\Gamma
= \cup_{i=1}^{n_f} \Gamma_i$ represents a fracture network consisting of $n_f$ surfaces. The
external boundary is subdivided into two
non-overlapping subsets, $\partial\Omega_u$ and $\partial\Omega_{\sigma}$, where Dirichlet and Neumann boundary conditions
apply, respectively. Each fracture $\Gamma_i$ consists of two overlapping surfaces,
$\Gamma_i^-$ and $\Gamma_i^+$, with the orientation defined by a unitary vector $\vec{n}_i$
orthogonal to the fracture plane. By convention, we choose $\vec{n}_i =
\vec{n}_i^- = -\vec{n}_i^+$. The pressure field is
defined on the union $\Gamma$ of the two-dimensional (2D) domains $\Gamma_i$, 
%
with $\partial \Gamma_i$ a one-dimensional (1D) curve defining the boundary of each fracture and $\overline{\Gamma}_i=\Gamma_i\cup\partial\Gamma_i$. The curve $\partial\Gamma_i$
is subdivided into two non-overlapping subsets, $\partial \Gamma_{i,p}$ and $\partial
\Gamma_{i,q}$, where Dirichlet and Neumann boundary conditions for the pressure
field are imposed. The vector $\vec{m}_i$ denotes the outer normal direction to
$\partial \Gamma_i$. The fluid is assumed to be incompressible, and body forces and
buoyancy effects are neglected.
The projection of the stress tensor $\tensorTwo{\sigma}$ along $\vec{n}_i$, 
$\vec{t} = \tensorTwo{\sigma} \cdot \vec{n}_i^- = -
\tensorTwo{\sigma} \cdot \vec{n}_i^+ = (t_N \vec{n}_i + \vec{t}_T)$, is the traction vector
over $\Gamma_i$, with $t_N$ and $\vec{t}_T$ 
its
normal and tangential component, respectively, with respect to the fracture-local reference frame.
The traction on $\Gamma_i$ controls the possible slipping and aperture of the fracture according to the Coulomb frictional law. A schematic representation of the considered conceptual framework is shown in Figure \ref{fig:conc_scheme}.

\begin{figure}
    \centering
    \null\hfill
    \subfloat[]{\includegraphics[width=0.45\linewidth]{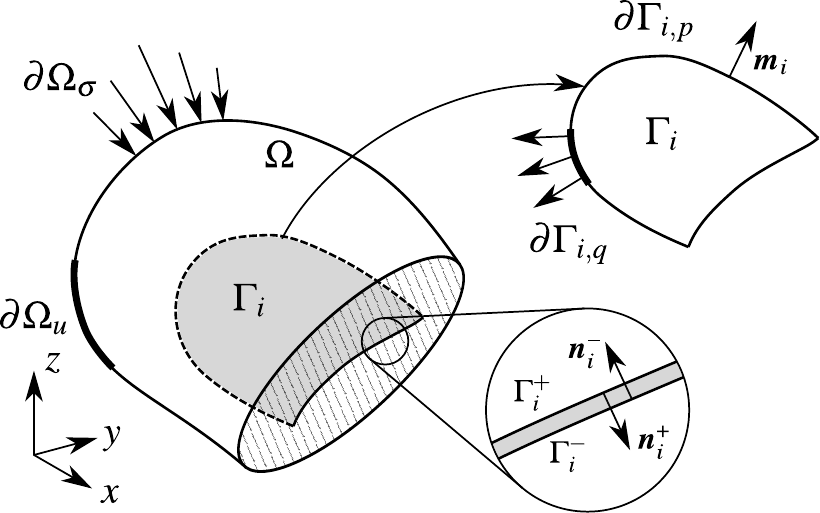}
      \label{fig:conc_scheme}}\hfill
    \subfloat[]{\includegraphics[width=0.45\linewidth]{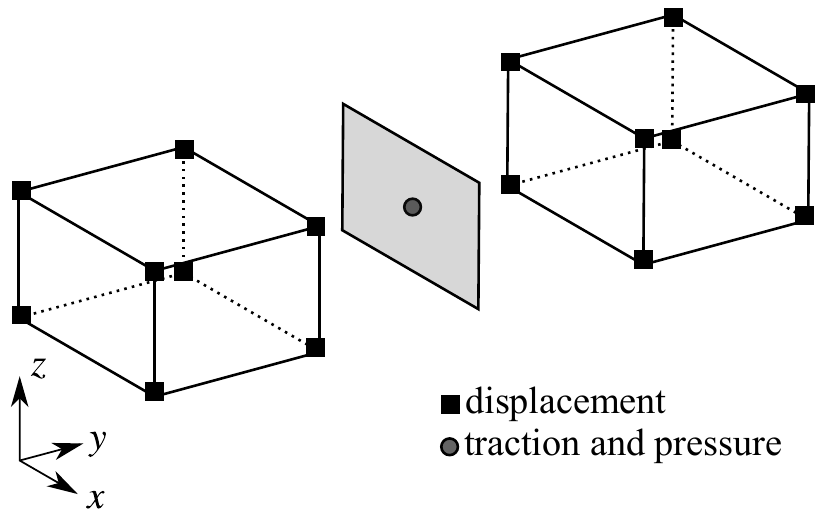}
      \label{fig:elem_scheme}}
    \hfill\null
    \caption{(a) Conceptual scheme of the elastic domain and embedded fracture network. (b) Example of low-order discretization.}
    \label{fig:scheme_elements}
\end{figure}

The strong form of the initial boundary value problem (IBVP) can be stated as follows
\cite{KikOde88,Lau03,Wri06,fr2020alg}: 
given the fluid discharge $q_s:\cup_{i=1}^{n_f}\Gamma_i\times\mathcal{T}\rightarrow\mathbb{R}$, the prescribed boundary displacement $\overline{\vec{u}}:\partial\Omega_u\times\mathcal{T}\rightarrow\mathbb{R}^3$ and traction $\overline{\vec{t}}:\partial\Omega_\sigma\times\mathcal{T}\rightarrow\mathbb{R}^3$, the prescribed fracture boundary pressure $\overline{p}:\cup_{i=1}^{n_f}\partial\Gamma_{i,p}\times\mathcal{T}\rightarrow\mathbb{R}$ and flux $\overline{q}:\cup_{i=1}^{n_f}\partial\Gamma_{i,q}\times\mathcal{T}\rightarrow\mathbb{R}$, the initial displacement $\vec{u}_0:\overline{\Omega}\rightarrow\mathbb{R}^3$ and pressure $p_0:\cup_{i=1}^{n_f}\overline{\Gamma}_i\rightarrow\mathbb{R}$,
find the displacement $\vec{u}:\overline{\Omega}\times[0,t_{\max}]\rightarrow\mathbb{R}^3$, the traction $\vec{t}:\cup_{i=1}^{n_f} \Gamma_i\times[0,t_{\max}]\rightarrow\mathbb{R}^3$, and pressure $p:\cup_{i=1}^{n_f} \overline{\Gamma}_i\times[0,t_{\max}]\rightarrow\mathbb{R}$ such that:

\begin{subequations}
\begin{align}
  - \nabla \cdot \tensorTwo{\sigma} (\vec{u}) &= \vec{0} & \mbox{ in }& \Omega \times \mathcal{T}
    &&& &\mbox{(linear momentum balance)}, \label{eq:momentumBalanceS}\\
  \dot{g}_N(\vec{u}) + \nabla \cdot \vec{q} (\tensorOne{u},p) &= q_s & \mbox{ in }&
    \cup_{i=1}^{n_f}\Gamma_i \times \mathcal{T} &&& &\mbox{(mass balance)}, \label{eq:massBalanceS}\\
  \vec{t} - p \vec{n}_i &= 0 & \mbox{ on }& \cup_{i=1}^{n_f}\Gamma_i \times \mathcal{T}
    &&& &\mbox{(traction balance)}, \label{eq:momentumBalanceS_F}\\
  \intertext{respecting the boundary conditions}
  \vec{u} &= \bar{\vec{u}} & \mbox{ on }& \partial \Omega_{u} \times \mathcal{T} &&& &\mbox{(prescribed
    boundary displacement)}, \label{eq:momentumBalanceS_DIR}\\
  \tensorTwo{\sigma}(\tensorOne{u}) \cdot \vec{n}_\Omega &= \bar{\vec{t}} & \mbox{ on }&
    \partial \Omega_{\sigma} \times \mathcal{T} &&& &\mbox{(prescribed boundary traction)},
    \label{eq:momentumBalanceS_NEU}\\
  p &= \bar{p} & \mbox{ on }& \cup_{i=1}^{n_f}\partial \Gamma_{i,p} \times \mathcal{T} &&& &\mbox{(prescribed boundary
    pressure)}, \label{eq:massBalanceS_DIR}\\
  \vec{q}(\tensorOne{u},p) \cdot \vec{m}_i &= \bar{q} & \mbox{ on }& \cup_{i=1}^{n_f}\partial \Gamma_{i,q} \times \mathcal{T} &&&
    &\mbox{(prescribed boundary flux)}, \label{eq:massBalanceS_NEU}\\
  \intertext{and initial conditions}
  \vec{u}|_{t=0} &= \vec{u}_0 & \mbox{ in }& \overline{\Omega} &&& &\mbox{(initial
    displacement)}, \label{eq:ICu}\\
  p|_{t=0} &= p_0 & \mbox{ in }& \cup_{i=1}^{n_f}\overline{\Gamma}_i &&& &\mbox{(initial pressure)},
   \label{eq:ICp}\\
  \intertext{subject to the constraints over each $\Gamma_i$ and for every time in $\mathcal{T}$}
  t_N = \vec{t} \cdot \vec{n}_i &\le 0, & g_N = \llbracket \vec{u} \rrbracket \cdot
    \vec{n}_i \ge 0, \quad t_N g_N &= 0 &&& &\mbox{(normal contact conditions)},
    \label{eq:normal_contact_KKT} \\
  \left\| \vec{t}_T \right\|_2 - \tau_{\max}(t_N) & \le 0, & \dot{\vec{g}}_T \cdot
    \vec{t}_T - \tau_{\max}(t_N) \left\| \dot{\vec{g}}_T \right\|_2 &= 0 &&& &
    \mbox{(Coulomb frictional law)}. \label{eq:frictional_contact_KKT}
\end{align}
\label{eq:IBVP}\null
\end{subequations}
In the problem statement, $\tensorTwo{\sigma}(\vec{u}) = \tensorFour{C}: \nabla^s \vec{u}$
is the Cauchy stress tensor, with $\tensorFour{C}$ the fourth-order elasticity tensor;
$\vec{q} (\vec{u},p) = -(C_f(\vec{u})/\mu) \nabla p$ is
the fluid volumetric flux in the fracture domain according to Darcy's law \cite{witherspoon1980validity}---assuming laminar flow---with $\nabla p$ the fluid pressure gradient,
$\mu$ the fluid viscosity (constant), and $C_f$ the isotropic fracture hydraulic
conductivity modeled as in \cite{GarKarTch16}:
\begin{equation}
  C_f = C_{f,0} + \frac{g_N^3}{12},
  \label{eq:cond_def}
\end{equation}
with $C_{f,0}$ the conductivity related to two irregular surfaces that are in contact
\cite{kamenov2013laboratory}; 
$\llbracket \vec{u} \rrbracket = ( \left.\vec{u}\right|_{ \Gamma_i^+ } - \left.\vec{u}
\right|_{ \Gamma_i^- } ) = (g_N \vec{n}_i + \vec{g}_T)$ denotes the relative displacement
across $\Gamma_i$, where $g_N$ and $\vec{g}_T$ are the normal and tangential components,
respectively, and $\left.\vec{u}\right|_{ \Gamma_i^+ }$ and $\left.\vec{u}\right|_{
\Gamma_i^- }$ are the restrictions of $\vec{u}$ on $\Gamma_i^+$ and $\Gamma_i^-$;
$\tau_{\text{max}} = c - t_N \tan(\theta)$ is the limit value provided by the static
Coulomb criterion, with $c$ and $\theta$ the cohesion and friction angle, respectively.
Since we employ a static Coulomb criterion, the tangential velocity
$\dot{\tensorOne{g}}_T$ in \eqref{eq:frictional_contact_KKT} is replaced with the tangential
displacement increment $\Delta{\tensorOne{g}}_T$ \cite{wohlmuth2011variationally} with
respect to the previously converged time-step. 

In our framework, we assume $\Gamma$ to be fixed with no propagation. The domain $\Gamma$ is
partitioned into three portions, where the following contact conditions occur:
\begin{itemize}
  \item \textit{stick} on $\Gamma^{\text{stick}}$: the fracture is closed
    ($\llbracket \vec{u} \rrbracket = \vec{0}$) and the traction vector $\vec{t}$ is unknown;
  \item \textit{slip} on $\Gamma^{\text{slip}}$: the fracture is closed in the
    normal direction ($g_N=0$ and $t_N$ is unknown), but a slip displacement $\vec{g}_T$ between
    $\Gamma_i^+$ and $\Gamma_i^-$ is allowed for, with $\vec{t}_T\vert_{\Gamma^{\text{slip}}} =
    \vec{t}_T^\ast = \tau_{\text{max}}(t_N) \frac{\vec{\Delta g}_T}{||\vec{\Delta g}_T||_2}$;
  \item \textit{open} on $\Gamma^{\text{open}}$: the fracture is fully open and a
    free relative displacement $\llbracket \vec{u} \rrbracket$ is allowed for, with $\vec{t}=\vec{0}$.
\end{itemize}
For additional details regarding the governing formulation, we refer the reader to
\cite{KikOde88,Lau03,Wri06,fr2020alg}.

\subsection{Discrete weak form}
In the solution to the model problem \eqref{eq:IBVP}, the traction $\vec{t}$ used as a primary variable plays the role of Lagrange multipliers.
Denoting with
$(\cdot,\cdot)_D$ the appropriate $L^2$-inner
product of scalar, vector or tensor functions in the spatial domain $D$, we introduce the finite-dimensional subspaces $\vecFunSpace{V}^h$,
$\vecFunSpace{M}^h(t^h_{N})$ and $\funSpace{P}^h$:
\begin{subequations}
  \begin{align}
      \vecFunSpace{V}^h \subset \vecFunSpace{V} &= \{ \tensorOne{\eta}\in[H^1(\Omega)]^3 : \tensorOne{\eta}=\overline{\vec{u}} \mbox{ on } \partial\Omega_u \}, \label{eq:Vh} \\
      \vecFunSpace{M}^h(t^h_{N}) \subset \vecFunSpace{M}(t_{N}) &= \left\{ \tensorOne{\mu} \in [L^2(\Gamma)]^3 : \mu_N \le 0, \left( \tensorOne{\mu}, \tensorOne{v} \right)_{\Gamma} \le \left( \tau_{\max}( t_N ), ||\tensorOne{v}_T||_2 \right)_{\Gamma}, \tensorOne{v} \in [H^{1/2}(\Gamma)]^3] \text{ with } v_N \ge 0 \right\}, \label{eq:Mh} \\
      \funSpace{P}^h \subset \funSpace{P} &= \{ \chi \in L^2(\Gamma) \}, 
      \label{eq:Ph}
  \end{align}
\end{subequations}
and the discrete approximations $\{\vec{u}^h, \vec{t}^h, p^h\}$ of $\{\vec{u}, \vec{t}, p\}$:
\begin{equation}
    \vec{u}^h = \sum_{i=1}^{n_u} \tensorOne{\eta_i} u_i \in \vecFunSpace{V}^h, \qquad \vec{t}^h = \sum_{j=1}^{n_t} \tensorOne{\mu}_j t_j \in \vecFunSpace{M}^h(t^h_N), \qquad p^h = \sum_{k=1}^{n_p} \chi_k p_k \in \funSpace{P}^h,
    \label{eq:discrete_approx}
\end{equation}
where, as before, the pedices $N$ and $T$ denote the components of a vector function along the normal and tangential direction with respect to a fracture-local reference frame on every $\Gamma_i$.
In \eqref{eq:discrete_approx}, $n_u$, $n_t$, and $n_p$ denote the number of discrete displacement, traction and pressure unknowns.
The weak form of
\eqref{eq:IBVP} reads \cite{fr2020alg}: find $\{\vec{u}^h, \vec{t}^h, p^h\} \in \vecFunSpace{V}^h \times
\vecFunSpace{M}^h(t^h_{N}) \times \funSpace{P}^h$ such that
\begin{subequations}
\begin{align}
  \mathcal{R}_{u} &= ( \nabla^s \tensorOne{ \eta }, \tensorTwo{\sigma} )_{\Omega}
    + \sum_{i=1}^{n_f} ( \llbracket \tensorOne{ \eta } \rrbracket, \tensorOne{ t }^h -
      p^h \tensorOne{n}_i )_{\Gamma_i}
    - ( \tensorOne{\eta}, \bar{\vec{t}} )_{\partial \Omega_{\sigma}}
    = 0 & & \forall \tensorOne{\eta} \in \vecFunSpace{V}^h_0, & \label{eq:weak_form_mom_h} \\
  \mathcal{R}_{t,i} &= ( t_N^h - \mu_N, g_N )_{\Gamma_i}
    + (\vec{t}_T^h - \tensorTwo{\mu}_T, \Delta {\vec{g}}_{T} )_{\Gamma_i} \ge 0 & & \forall \tensorOne{\mu} \in \vecFunSpace{M}^h (t^h_N), & i=1,\ldots,n_f,
    \label{eq:weak_lam_h}\\
  \mathcal{R}_{p,i} &= \left( \chi, \frac{\Delta g_{N} }{\Delta t} \right)_{\Gamma_i}
    + [ \chi, p^h ]_{\mathcal{F}_i} - F_{\mathcal{F}_i}(\chi) + G_{\mathcal{F}_i}(\chi)
    - ( \chi, q_s )_{\Gamma_i} = 0 & & \forall \chi \in \funSpace{P}^h, & i=1,\ldots,n_f, \label{eq:weak_form_mass_h}
\end{align}
\label{eq:weak_form_discr}\null
\end{subequations}
where $\vecFunSpace{V}^h_0$ is $\vecFunSpace{V}^h$ with homogeneous conditions along $\partial\Omega_u$, $\Delta t$ is the time step size, and $[\chi, p^h]_{\mathcal{F}}$ is a weighted inner product representing the classical
two-point flux approximation (TPFA) scheme. This is introduced to allow a unified
presentation of the coupled finite element/finite volume model \cite{EymGalHer00,
EymGalHer07,Age_etal10}. In particular, we have:
\begin{align}
  [\chi, p^h]_{\mathcal{F}} = 
    \sum_{\mathcal{E}_{\text{int}}} (\chi\vert_{\varphi_L} - \chi\vert_{\varphi_K} )
    \Upsilon_{KL} ( p^h\vert_{ \varphi_L } - p^h\vert_{ \varphi_K } ) +
    \sum_{\mathcal{E}_p} \chi\vert_{ \varphi_K } \Upsilon_K p^h\vert_{ \varphi_K },
    \label{eq:FV_inner_product}
\end{align}
where $\mathcal{E}_{\text{int}}$ and $\mathcal{E}_p$ represent the set of edges included
in $\Gamma$ and $\partial\Gamma_{p}$, respectively; $\varphi_K$ and $\varphi_L$ are the
two adjacent cells $K$ and $L$; and $\Upsilon_{KL}$ is the harmonic average of one-sided
transmissibility $\Upsilon_K$ and $\Upsilon_L$ associated to $\varphi_K$ and $\varphi_L$.
Finally, $F_{\mathcal{F}}$ and $G_{\mathcal{F}}$ collect the boundary conditions to be
prescribed on $\partial\Gamma_{p}$ and $\partial\Gamma_{q}$, respectively. For further
details, we refer the reader to \cite{fr2020alg}.

To solve the problem \eqref{eq:weak_form_discr}, we
transform the variational inequality \eqref{eq:weak_lam_h} into a variational equality. For this
purpose, we apply an active-set algorithm, as described in \cite{nocedal2006numerical,
antil2018frontiers, fr2020alg}, which allows to identify
the subdivision into stick/slip/open regions for every $\Gamma_i$. At a given step of the active-set algorithm, the stick/slip/open regions of each fracture $\Gamma_i$ are fixed and the inequality \eqref{eq:weak_lam_h} becomes:
\begin{equation}
  \mathcal{R}_{t,i}
  = \left( \tensorOne{\mu}, \tensorOne{g} \right)_{\Gamma_{i}^{\text{stick}} }
  + \left( \mu_N, g_N\right)_{\Gamma_{i}^{\text{slip}}}
  + \frac{1}{k} \left( \tensorTwo{\mu}_T, \tensorOne{t}_{T}^h - \tensorOne{t}_T^*
    \right)_{\Gamma_{i}^{\text{slip}}}
  + \frac{1}{k} \left( \tensorTwo{\mu}, \tensorTwo{t}^h
    \right)_{\Gamma_{i}^{\text{open}}}
  = 0, \qquad i=1,\ldots,n_f,
  \label{eq:weak_lam_partitioned}
\end{equation}
with $k$ a coefficient needed to ensure the dimensional consistency of the equation. 
Introducing in \eqref{eq:weak_form_mom_h}, \eqref{eq:weak_lam_partitioned} and \eqref{eq:weak_form_mass_h} the finite-dimensional bases of $\vecFunSpace{V}^h_0$,
$\vecFunSpace{M}^h( t_N^h )$ and $\funSpace{P}^h$ yields the
following system of nonlinear discrete residual equations:
\begin{equation}
  \blkVec{r} (\Vec{u}^{\ell}, \Vec{t}^{\ell}, \Vec{p}^{\ell})
  = \begin{bmatrix}
      \blkVec{r}_u (\Vec{u}^{\ell}, \Vec{t}^{\ell}, \Vec{p}^{\ell})\\
      \blkVec{r}_t (\Vec{u}^{\ell}, \Vec{t}^{\ell}, \Vec{p}^{\ell})\\
      \blkVec{r}_p (\Vec{u}^{\ell}, \Vec{t}^{\ell}, \Vec{p}^{\ell})
    \end{bmatrix}
  = \blkVec{0},
  \label{eq:discrete_res}
\end{equation}
which is solved by a Newton-Krylov method. In \eqref{eq:discrete_res}, the algebraic vectors
$\Vec{u}^{\ell}\in\mathbb{R}^{n_u}$, $\Vec{t}^{\ell}\in\mathbb{R}^{n_t}$ and $\Vec{p}^{\ell}\in\mathbb{R}^{n_p}$ collect the coefficients $u_i$,
$t_j$ and $p_k$ of the discrete displacement, traction and pressure fields in \eqref{eq:discrete_approx}
and $\ell$ is the active-set counter. After convergence of the Newton-Krylov method at the $\ell$-th step of the active-set algorithm, a consistency check is carried out in order to verify whether the assumed stick/slip/open region subdivision meets the Coulomb frictional conditions. If not, the region subdivision is updated and a new step is performed. The algorithm stops when the consistency check does not require to modify the stick/slip/open region subdivision. At this point, convergence is achieved and the solution is sought at the following time step. 

The finite element/finite volume spaces used in this work are the same as in
\cite{fr2020alg}, i.e., first-order continuous finite elements for displacements and
face-centered piecewise-constant elements for tractions and pressures, as schematically represented in Figure \ref{fig:elem_scheme}. 
To model the fractures, we use a DFM approach
with a conforming mesh~\cite{GarKarTch16}, hence any $\Gamma_i$ is
represented by a set of finite element faces.
Thus,
displacement unknowns are located on mesh vertices, while traction and pressure unknowns are
on fracture faces (Figure \ref{fig:elem_scheme}), with $n_u$ equal to three times the number of 3D finite element nodes, $n_p$ equal to the number of 2D faces discretizing the fracture network, and $n_t$ equal to $3\cdot n_p$. Displacements are represented in the global reference system and
tractions are represented in a face-based local reference frame. 
This approach is intrinsically unstable, as
it does not fulfill the inf-sup condition \cite{wohlmuth2011variationally}. In this work,
we use the global algebraic stabilization proposed in \cite{fr2020alg}, which relaxes the
zero jump and the impenetrability conditions between the two fracture surfaces in the traction balance equation, and the fluid incompressibility constraint in the mass balance equation.
Only stick and slip portions are involved in the traction balance, being the tractions in the open part known.
With the introduction of the stabilization, equations
\eqref{eq:weak_lam_partitioned} and \eqref{eq:weak_form_mass_h} become:
\begin{subequations}
\begin{align}
  \mathcal{R}_{t,i}
  &= \left( \tensorOne{\mu}, \tensorOne{g} \right)_{\Gamma_{i}^{\text{stick}} }
  + \left( \mu_N, g_N\right)_{\Gamma_{i}^{\text{slip}}}
  + \frac{1}{k}\left( \tensorTwo{\mu}_T, \tensorOne{t}_{T}^h - \tensorOne{t}_T^*
    \right)_{\Gamma_{i}^{\text{slip}}}
  + \frac{1}{k}\left( \tensorTwo{\mu}, \tensorTwo{t}^h
    \right)_{\Gamma_{i}^{\text{open}}}
  - j_t \left( \tensorTwo{\mu}, \tensorTwo{t}^h \right)_{\Gamma_i^{\text{stick}}
    \cup\Gamma_i^{\text{slip}}}
  = 0, & &i=1,\ldots,n_f,
  \label{eq:weak_lam_partitioned_stab} \\
  \mathcal{R}_{p,i} &= \left( \chi, \frac{\Delta g_{N} }{\Delta t} \right)_{\Gamma_i}
    + [ \chi, p^h ]_{\mathcal{F}_i} - F_{\mathcal{F}_i}(\chi) + G_{\mathcal{F}_i}(\chi)
    - ( \chi, q_s )_{\Gamma_i} + \frac{1}{\Delta t} j_p \left( \chi, p^h \right)_{\Gamma_i} = 0, & & i=1,\ldots,n_f, \label{eq:weak_mass_partitioned_stab}
  \end{align}
  \label{eq:weak_stab}\null
\end{subequations}
where $j_t ( \tensorTwo{\mu}, \tensorTwo{t}^h )$ and $j_p ( \chi, p^h )$ are the stabilizing bilinear forms
for the traction and pressure field, respectively.
In particular, we have
\begin{equation}
  j_t \left( \tensorTwo{\mu}, \tensorTwo{t}^h \right)
  = \sum_{\epsilon \in \mathcal{E}_{\text{int}}}
    \frac{1}{|\epsilon|}
    \int_{\epsilon}
    \llbracket \tensorTwo{\mu} \rrbracket_{\epsilon}
    \cdot
    \tensorTwo{\Upsilon}^{(\epsilon)}
    \cdot
    \llbracket \tensorTwo{t}^h \rrbracket_{\epsilon}
    \; \mathrm{d}l,
    \label{eq:jump_bilinear_form}
\end{equation}
with $\llbracket \cdot \rrbracket_{\epsilon}$ denoting the jump of a quantity across
the generic internal edge $\epsilon$ and $\tensorTwo{\Upsilon}^{(\epsilon)}$ is a
positive definite second-order tensor providing the appropriate scaling. The discrete
formulation of $\tensorTwo{\Upsilon}^{(\epsilon)}$ is fully provided in
\cite{fr2020alg}. 
The contribution $j_p ( {\chi}, {p}^h ) $ is computed as the normal projection of $j_t$ with respect to the surface $\Gamma_i$.

At a given active-set iteration $\ell$, the Newton linearization of \eqref{eq:discrete_res}, which
now includes also the stabilization terms, generates a sequence of linear
systems and vector updates. To advance by one Newton iteration $k$, we have to:
  \begin{equation}
    \begin{aligned}
      &\text{solve} \; \mathcal{J}^{\ell,(k)} \delta \mathbf{x} = -\mathbf{r}^{\ell,(k)} 
      && \Rightarrow \qquad
      \begin{bmatrix}
        \Mat{A} & \Mat{C}_1 & \Mat{Q}_1 \\
        \Mat{C}_2 & -\Mat{H} & \Mat{0} \\
        \Mat{Q}_2 & \Mat{0} & \Mat{T} \\
      \end{bmatrix}^{\ell,(k)}
      \begin{bmatrix}
        \delta \Vec{u} \\
        \delta \Vec{t} \\
        \delta \Vec{p}
      \end{bmatrix}
      =
      -\begin{bmatrix}
        \Vec{r}_u \\
        \Vec{r}_t \\
        \Vec{r}_p
      \end{bmatrix}^{\ell,(k)}, \\
      &\text{update} \; \mathbf{x}^{\ell,(k+1)} = \mathbf{x}^{\ell,(k)} + \delta \mathbf{x}
      && \Rightarrow \qquad
      \begin{bmatrix}
        \Vec{u} \\
        \Vec{t} \\
        \Vec{p}
      \end{bmatrix}^{\ell,(k+1)}
      =
      \begin{bmatrix}
        \Vec{u} \\
        \Vec{t} \\
        \Vec{p}
      \end{bmatrix}^{\ell,(k)}
      +
      \begin{bmatrix}
        \delta \Vec{u} \\
        \delta \Vec{t} \\
        \delta \Vec{p}
      \end{bmatrix}.
    \end{aligned}
    \label{eq:JacSys}
  \end{equation}
The submatrices in the $3 \times 3$ block Jacobian read:
  \small
  \begin{subequations}
    \begin{align}
      [\Mat{A}]_{ij} &
      = \left(\nabla^s \tensorOne{\eta}_i,\tensorFour{C}:\nabla^s
        \tensorOne{\eta}_j\right)_{\Omega}, && i=1,n_u, && j=1,n_u,
      \\
      [\Mat{C}_1]_{ij} &
      = \left( \llbracket \tensorOne{ \eta }_i \rrbracket, \tensorOne{\mu}_j
        \right)_{\Gamma}, && i=1,n_u, && j=1,n_t,
      \\
      [\Mat{Q}_1]_{ij} &
      = -\left( \llbracket \eta_{i,N} \rrbracket, \chi_j \right)_{\Gamma}, && i=1,n_u, && j=1,n_p,
      \label{eq:Q1def} \\
      [\Mat{C}_2]_{ij} &
      = \left( \tensorOne{\mu}_i, \llbracket \tensorOne{\eta}_j \rrbracket
        \right)_{\Gamma^{\ell,\text{stick}}}
      + \left( {\mu}_{i,N}, \llbracket {\eta}_{j,N} \rrbracket
        \right)_{\Gamma^{\ell,\text{slip}}}
      - \frac{1}{k}\left( \tensorOne{\mu}_{i,T}, \left( \frac{\partial \tensorOne{t}_T^*}
        {\partial \Delta \tensorOne{g}_T} \right) \biggr|^{\ell,(k)} \cdot \llbracket
        \tensorOne{\eta}_{j,T} \rrbracket \right)_{\Gamma^{\ell,\text{slip}}}, && i=1,n_t, && j=1,n_u,
      \label{eq:C2def} \\
      [\Mat{H}]_{ij} &
      = -\frac{1}{k}\left( \tensorOne{\mu}_{i,T}, \tensorOne{\mu}_{j,T}
        \right)_{\Gamma^{\ell,\text{slip}}}
      + \frac{1}{k}\left( \tensorOne{\mu}_{i,T}, \left( \frac{\partial \tensorOne{t}_T^*}
        {\partial t_N} \right) \biggr|^{\ell,(k)} \mu_{j,N} \right)_{\Gamma^
        {\ell,\text{slip}}}
      - \frac{1}{k}\left( \tensorOne{\mu}_{i}, \tensorOne{\mu}_{j}
        \right)_{\Gamma^{\ell,\text{open}}}
      + j_t\left( \tensorOne{\mu}_{i}, \tensorOne{\mu}_{j} \right)_{\Gamma^{\ell,\text{stick}}\cup\Gamma^{\ell,\text{slip}}}, && i=1,n_t, && j=1,n_t,
      \label{eq:Hdef} \\
      [\Mat{Q}_2]_{ij} &
      = \frac{1}{\Delta t} \left( \chi_i, \llbracket \eta_{j,N} \rrbracket
        \right)_{\Gamma^{\ell,\text{open}} }
      + \left.\frac{\partial \left( [ \chi_i, p^h ]_{\mathcal{F}} \right) }{\partial u_j }
        \right|^{\ell,(k)}
      - \left.\frac{\partial \left( F_{\mathcal{F}}( \chi_i ) \right) }{\partial u_j }
        \right|^{\ell,(k)}, && i=1,n_p, && j=1,n_u,
      \label{eq:Q2def} \\
      [\Mat{T}]_{ij} &
      = \left.\frac{\partial \left( [ \chi_i, p^h ]_{\mathcal{F}} \right) }{\partial p_j }
        \right|^{\ell,(k)}
      + \frac{1}{\Delta t} j_p\left( {\chi}_{i}, {\chi}_{j} \right)_{\Gamma^{\ell}}, && i=1,n_p, && j=1,n_p.
    \end{align}
    \label{eq:block_def}\null
  \end{subequations}
%
The partial derivatives appearing in \eqref{eq:block_def} are reported in
\cite[Appendix A]{fr2020alg}.

\subsection{Linear system}
We focus our attention on the linear system solution and
the design of robust, scalable and efficient preconditioners for the $3 \times 3$
block matrix of equation \eqref{eq:JacSys}. 
The global matrix $\mathcal{J}$ is large, sparse, and non-symmetric, with properties that change with the evolution of the stick/slip/open regions in the fracture network. A representative evolution of the non-zero pattern of $\mathcal{J}$ during a full simulation is shown in Figure \ref{fig:patterns}.
The features that follow are worth summarizing.
\begin{enumerate}
  \item The first block row of $\mathcal{J}$ includes the contributions arising from the linear momentum balance of the 3D domain $\Omega$. All the submatrices do not depend on the fracture state and can be assembled once at the beginning of the whole simulation if an elastic constitutive law is used. In particular, $\Mat{A}$ is the classical symmetric positive definite (SPD) elastic stiffness matrix, while 
    $\Mat{C}_1$ and $\Mat{Q}_1$ are tall rectangular blocks collecting a surface measure of the fracture elements and transferring tractions and pressures to the 3D body as applied forces. 
  \item In the second block row of $\mathcal{J}$, $\Mat{C}_2$ varies as the stick/slip/open fracture regions evolve through the active-set algorithm, in both the entry values and the non-zero pattern (Figure \ref{fig:patterns}).
  If all the fracture elements are in stick mode, we have that $\Mat{C}_2 = \Mat{C}_1^T$, otherwise
  the frictional law derivatives appear and $\Mat{C}_2\neq\Mat{C}_1^T$.
  \item When all fractures belong to the stick region, $\Mat{H}$ is the symmetric positive semidefinite (SPSD) stabilization matrix. In case of sliding,
non-symmetric diagonal $2 \times 2$ blocks arise, one for each traction component along the local tangential direction to the fracture surface. In the open regions the rows of $H$ have a single non-zero entry in the main diagonal, with no contribution from the stabilization term (Figure \ref{fig:patterns}). In any case, $H$ is singular and cannot be regularly inverted.
  \item The third block row of $\mathcal{J}$ includes the contributions arising from the fluid mass balance on the fracture network. The coupling between fluid flow and fracture mechanics is controlled by
  $\Mat{Q}_2$. In particular, when all fracture elements are in stick mode, $\Mat{Q}_2 = \Mat{0}$
  and $\mathcal{J}$ is reducible with a $2\times2$ symmetric saddle-point matrix as leading block. Otherwise, 
   contributions from the flux derivative with respect to
the displacements appear, i.e., $\Mat{Q}_2$ entries depend on the current pressure solution (Figure \ref{fig:patterns}). By distinction with $\Mat{C}_1^T$ and $\Mat{C}_2$, there is no simple relationship between $\Mat{Q}_1^T$ and $\Mat{Q}_2$, in both the entry values and the non-zero pattern. Denoting with $\Psi$ the matrix-to-matrix
operator returning a zero row if the corresponding element index belongs to
$\Gamma^{\text{stick}} \cup \Gamma^{\text{slip}}$ and the original row if the element index belongs to $\Gamma^{\text{open}}$,
$\Mat{Q}_2$ can be written as:
\begin{equation}
  \Mat{Q}_2 = -\frac{\Psi\left( \Mat{Q}_1^T \right)}{\Delta t} + \Mat{F}_u,
  \label{eq:Q2vsQ1}
\end{equation}
where $\Mat{F}_u$ collects the contributions from the flux derivatives with respect to the displacements.
  \item $\Mat{T}$ is the sum of the standard transmissibility matrix arising from the TPFA discretization in the 2D domain $\Gamma$ and the
    stabilization contribution. As such, it is SPD with the 5-point stencil of a 2D discrete Laplacian. Moreover, $\Mat{T}$ has a block diagonal structure for all non-intersecting fractures. Observe also that traction and pressure fields are always decoupled.
\end{enumerate}
From the observations above, it appears that matrix $\mathcal{J}$ changes nature with the evolution of the fracture conditions, moving from a reducible matrix with a symmetric saddle-point leading block to a general non-symmetric and indefinite matrix. The objective of our work is to define a unique preconditioning framework ensuring robustness, scalability and computational efficiency for any working situation.

\begin{figure}
  \centering
  \null\hfill
  \subfloat[Pure stick mode]{\includegraphics[width=0.33\linewidth]{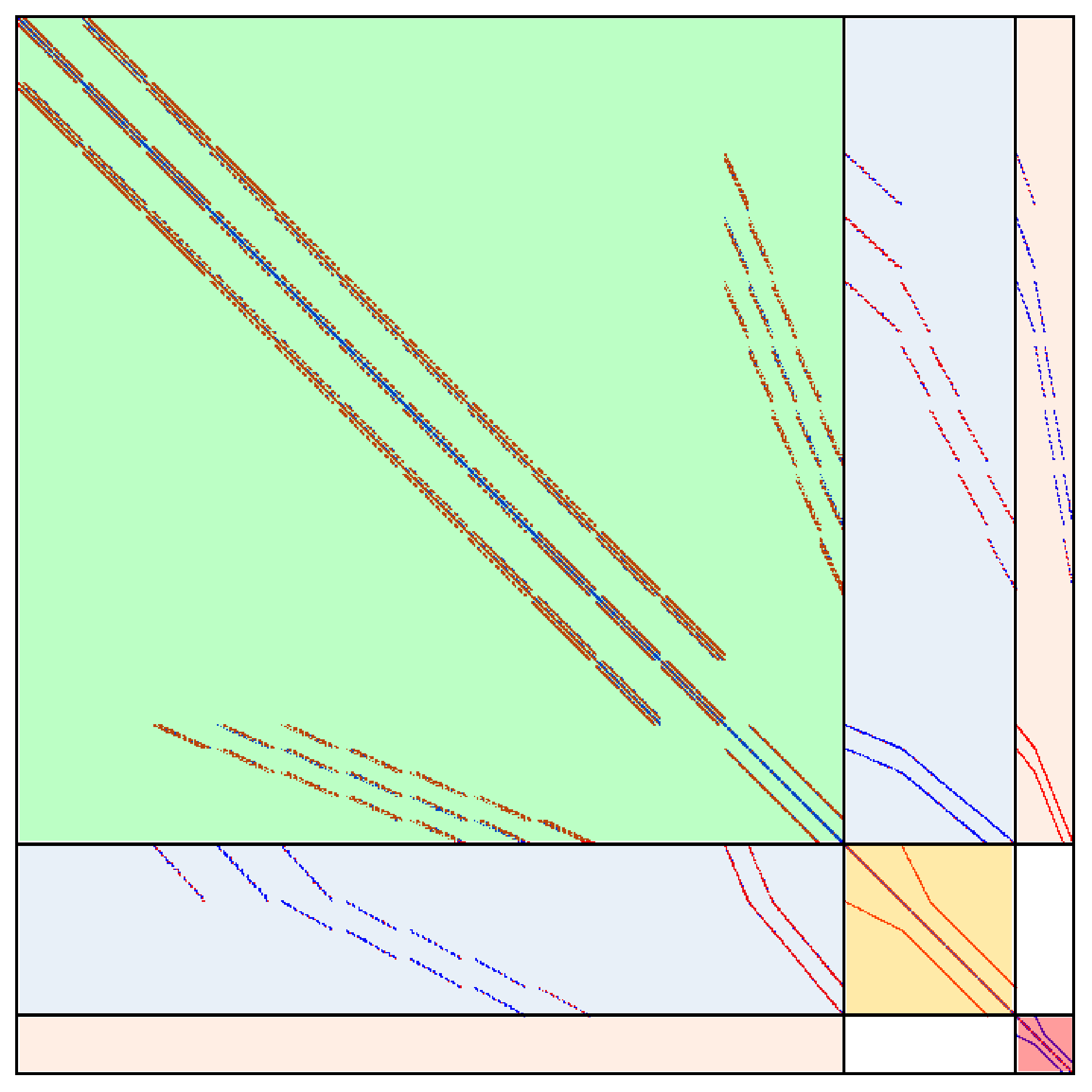}}\hfill
  \subfloat[Stick/slip/open modes]{\includegraphics[width=0.33\linewidth]{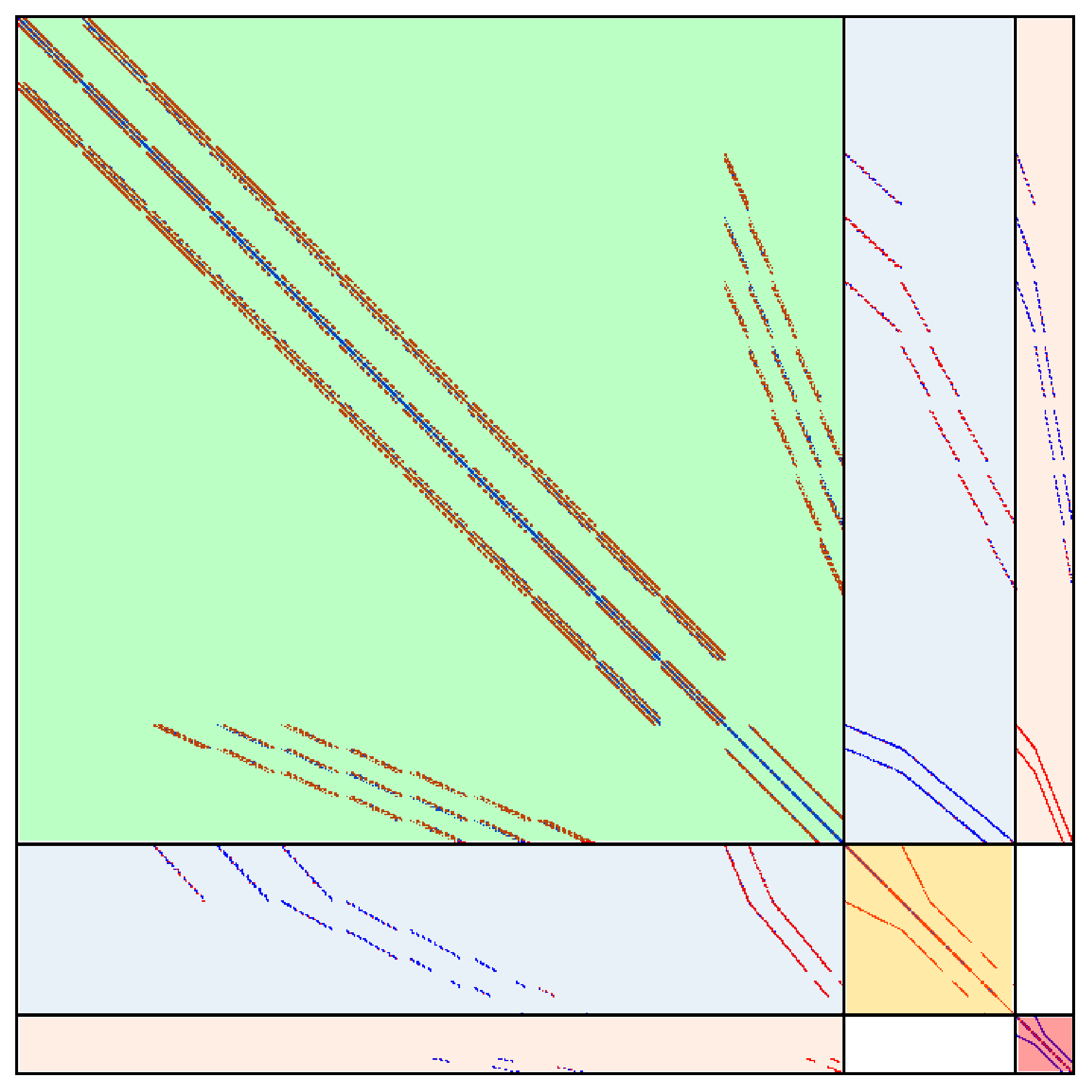}}\hfill
  \subfloat[Pure open mode]{\includegraphics[width=0.33\linewidth]{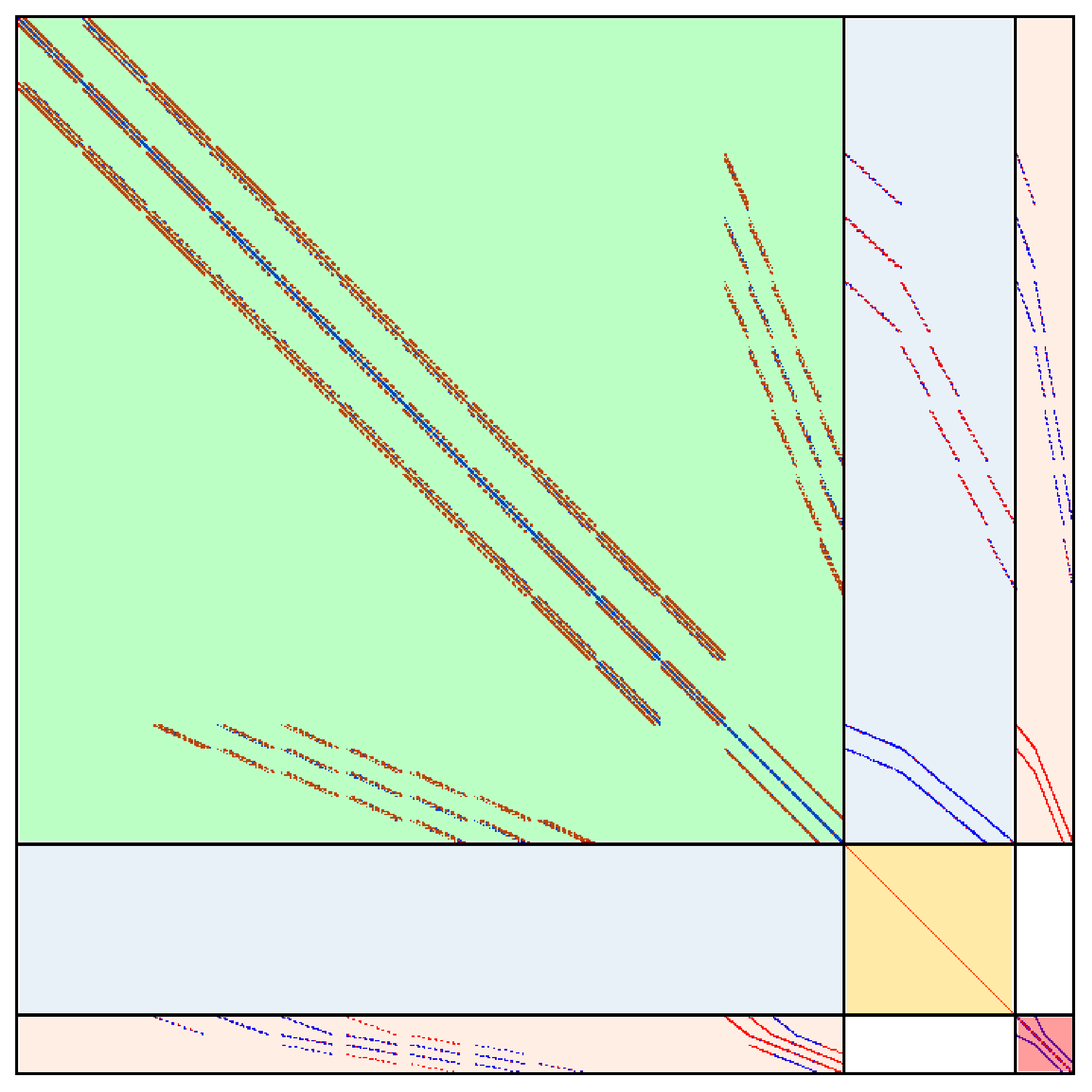}}
  \hfill\null
  \caption{Non-zero pattern of $\mathcal{J}$ with the evolution of the stick/slip/open fracture regions.}
  \label{fig:patterns}
\end{figure}

\section{Preconditioning framework}
\label{sec:preconditioner}

A preconditioner $\mathcal{M}^{-1}$ of $\mathcal{J}$ is a non-singular operator whose application to a vector resembles as much as possible the action of $\mathcal{J}^{-1}$. The exact application of $\mathcal{J}^{-1}$ to some vector $\mathbf{w}\in\mathbb{R}^{n_u+n_t+n_p}$ provides the vector $\mathbf{v}\in\mathbb{R}^{n_u+n_t+n_p}$ such that:
\begin{equation}
    \left\{ \begin{array}{ccccccc}
    A\mathbf{v}_u &+& C_1\mathbf{v}_t &+& Q_1\mathbf{v}_p &=& \mathbf{w}_u \\
    C_2\mathbf{v}_u &-& H\mathbf{v}_t & & &=& \mathbf{w}_t \\
    Q_2\mathbf{v}_u & & &+& T\mathbf{v}_p &=& \mathbf{w}_p \end{array} \right. ,
    \label{eq:ex_Jac_appl}
\end{equation}
with $\mathbf{v}_u,\mathbf{w}_u\in\mathbb{R}^{n_u}$, $\mathbf{v}_t,\mathbf{w}_t\in\mathbb{R}^{n_t}$, and $\mathbf{v}_p,\mathbf{w}_p\in\mathbb{R}^{n_p}$ natural subvectors of $\mathbf{v},\mathbf{w}$, respectively. The objective is to approximate the solution to the multi-physics system \eqref{eq:ex_Jac_appl} by exploiting the physics-based variable partitioning. The system is first reduced to a single-physics problem, and then prolonged back to the full multi-physics space. According to the selected sequence of reductions, different algorithms may arise. 

\subsection{Method no. 1: t-p-u approach}

Traction and pressure variables live on the fractures and are mutually decoupled independently on the stick/slip/ open region partitioning. Therefore, it is natural to exploit this condition and perform a simultaneous reduction of both variable sets onto the displacement space. This corresponds to compute $\mathbf{v}_t$ and $\mathbf{v}_p$ from the second and third equation of \eqref{eq:ex_Jac_appl}, respectively, and introduce them in the first equation, thus eliminating both physics from the equilibrium equation. 
Recall, however, that
$\Mat{H}$ is singular, so a regular surrogate is
needed. 
A block diagonal approximation can be used instead, where each block is the $3\times3$ local stabilization matrix computed for each fracture element.
Denoting with $\Mat{\tH}$ such a block-diagonal approximation, we have:
\begin{subequations}
\begin{align}
    \mathbf{v}_t &\simeq -\Mat{\tH}^{-1} \left( \mathbf{w}_t - \Mat{C}_2 \mathbf{v}_u \right), \label{eq:t_reduc1} \\
    \mathbf{v}_p &= \Mat{T}^{-1} \left( \mathbf{w}_p - \Mat{Q}_2 \mathbf{v}_u \right). \label{eq:p_reduc1}
\end{align}
\label{eq:tp_reduc1}\null
\end{subequations}
With \eqref{eq:t_reduc1} and \eqref{eq:p_reduc1}, the first equation of \eqref{eq:ex_Jac_appl} becomes:
\begin{equation}
    \left( A + C_1 \Mat{\tH}^{-1} C_2 - Q_1 T^{-1} Q_2 \right) \mathbf{v}_u \simeq \mathbf{w}_u + C_1 \Mat{\tH}^{-1} \mathbf{w}_t - Q_1 T^{-1} \mathbf{w}_p,
    \label{eq:u_1}
\end{equation}
which is a single-physics equilibrium equation on the 3D domain where the elimination of fracture tractions and pressures introduces fictitious stiffness contributions. The matrix at the left-hand side of \eqref{eq:u_1} is the Schur complement $S$: 
\begin{equation}
    S = A + C_1 \Mat{\tH}^{-1} C_2 - Q_1 T^{-1} Q_2.
    \label{eq:Schur_1}
\end{equation}
Solution to \eqref{eq:u_1} provides $\mathbf{v}_u$, which, introduced into equations \eqref{eq:tp_reduc1}, yields the final vector $\mathbf{v}$. The multi-physics reduction order performed in this case is traction-pressure-displacement (t-p-u) and is schematically summarized in Figure \ref{fig:MPReduction}.

\begin{figure}
    \centering
    \includegraphics[width=0.6\linewidth]{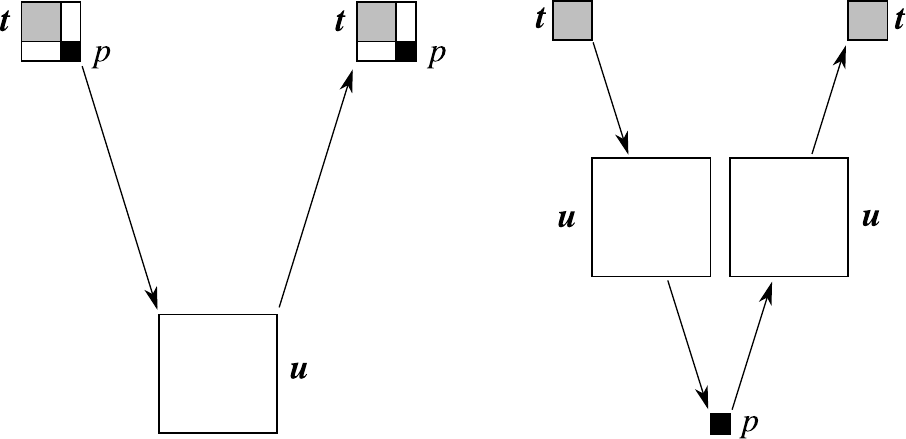}
    \caption{Schematic representation of the multi-physics reduction preconditioning framework: t-p-u (left) and t-u-p (right) approach.}
    \label{fig:MPReduction}
\end{figure}

The computation and inversion of $S$ in \eqref{eq:Schur_1} cannot be performed exactly. The Schur complement is explicitly approximated by $\tS$:
\begin{equation}
    S \simeq \tS = A + C_1 \Mat{\tH}^{-1} C_2 - Q_1 \tT^{-1} Q_2,
    \label{eq:appSchur_1}
\end{equation}
where $\tT$ is a diagonal surrogate for $T$. The inverse $\tS^{-1}$ can be applied inexactly by means of an AMG operator, which can be efficiently used in mechanical problems preserving a linear complexity with respect to the problem size. This is a key property to guarantee the solver scalability. Recent examples of effective AMG preconditioners are, for instance, taken from the References \cite{brandt2014bootstrap,
dambra2018bootcmatch, dambra2019improving, paludetto2019novel}. In this work, we use an aggregation-based multigrid as
the reference AMG operator. Specifically, the application of $\tS^{-1}$ is approximated by GAMG \cite{may2016extreme}, the
state-of-the-art aggregation based multigrid provided by the PETSc package
\cite{petsc-user-ref}. 

The construction and application of the resulting preconditioning operator $\mathcal{M}_1^{-1}$ with the t-p-u approach is summarized in Algorithms~\ref{alg:cpt_tpu} and~\ref{alg:app_tpu}. The $\text{Bdiag}(K,n)$ operator gives a matrix with the $n\times n$ diagonal blocks of $K$, while the $\text{AMG}(K,\mathbf{y})$ operator applies the selected AMG preconditioner of $K$ to the vector $\mathbf{y}$. Since $n_u$ is generally much larger than $n_t$ and $n_p$, the cost for applying the exact inverse of $\tH$ and $T$ is negligible with respect to the AMG algorithm for $\tS$. Hence, the latter can be roughly assumed as the cost per iteration for the $\mathcal{M}_1^{-1}$ application. 

\begin{algorithm}
\caption{ {\sc Preconditioner Construction: t-p-u approach}
[$\tH$, $\tS$]=cpt\_tpu($\mathcal{J}$).}
\label{alg:cpt_tpu}
\begin{algorithmic}[1]
\State $\tH = \text{Bdiag}(H,3)$;
\State $\tT = \text{diag}(T)$;
\State $\tS = A + C_1 \tH^{-1} C_2 - Q_1 \tT^{-1} Q_2$;
\end{algorithmic}
\end{algorithm}

\begin{algorithm}
\caption{ {\sc Preconditioner Application: t-p-u approach}
[$\mathbf{v}_u$, $\mathbf{v}_t$, $\mathbf{v}_p$]=app\_tpu($\mathcal{J}$, $\tH$, $\tS$, $\mathbf{w}_u$, $\mathbf{w}_t$, $\mathbf{w}_p$).}
\label{alg:app_tpu}
\begin{algorithmic}[1]
\State $\mathbf{t}_p = T^{-1} \mathbf{w}_p$;
\State $\mathbf{t}_t = \tH^{-1} \mathbf{w}_t$;
\State $\mathbf{t}_u = \mathbf{w}_u + C_1 \mathbf{t}_t - Q_1 \mathbf{t}_p$;
\State $\mathbf{v}_u = \text{AMG}(\tS, \mathbf{t}_u)$;
\State $\mathbf{s}_p = Q_2 \mathbf{v}_u$;
\State $\mathbf{s}_t = C_2 \mathbf{v}_u$;
\State $\mathbf{v}_p = \mathbf{t}_p - T^{-1} \mathbf{s}_p$;
\State $\mathbf{v}_t = - \mathbf{t}_t + \tH^{-1} \mathbf{s}_t$;
\end{algorithmic}
\end{algorithm}

From an algebraic viewpoint, the preconditioning operator $\mathcal{M}_1^{-1}$ arising from the t-p-u approach can be written as an inexact block LDU factorization of $\mathcal{J}$. 
Using the permutation matrix $\mathcal{Q}_1$:
\begin{equation}
  \mathcal{Q}_1 = \begin{bmatrix}
    \Mat{0} & \Mat{I}_t & \Mat{0} \\
    \Mat{0} & \Mat{0} & \Mat{I}_p \\
    \Mat{I}_u & \Mat{0} & \Mat{0} \\
  \end{bmatrix},
  \label{eq:Q1}
\end{equation}
where $\Mat{I}_n$ is the identity matrix in $\mathbb{R}^{n_n\times n_n}$ and $\Mat{0}$ the zero matrix of proper size,
the block LDU factorization reads:
\begin{equation}
  \mathcal{Q}_1 \mathcal{J} \mathcal{Q}_1^T \simeq \mathcal{L}_1 \mathcal{D}_1 \mathcal{U}_1,
\end{equation}
with:
\begin{align}
  \mathcal{L}_1 &= \begin{bmatrix}
    \Mat{I}_t & \Mat{0} & \Mat{0} \\
    \Mat{0} & \Mat{I}_p & \Mat{0} \\
    -\Mat{C}_1\Mat{\tH}^{-1} & \Mat{Q}_1\Mat{T}^{-1} & \Mat{I}_u \\
  \end{bmatrix}, &&&
  \mathcal{D}_1 &= \begin{bmatrix}
    -\Mat{\tH} & \Mat{0} & \Mat{0} \\
    \Mat{0} & \Mat{T} & \Mat{0} \\
    \Mat{0} & \Mat{0} & \Mat{\tS} \\
  \end{bmatrix} & \text{and} &&
  \mathcal{U}_1 &= \begin{bmatrix}
    \Mat{I}_t & \Mat{0} & -\Mat{\tH}^{-1}\Mat{C}_2 \\
    \Mat{0} & \Mat{I}_p & \Mat{T}^{-1}\Mat{Q}_2 \\
    \Mat{0} & \Mat{0} & \Mat{I}_u \\
  \end{bmatrix}.
  \label{eq:LDU1}
\end{align}
Hence, the final algebraic expression of $\mathcal{M}_1^{-1}$ is:
\begin{equation}
  \mathcal{M}_1^{-1} = \begin{bmatrix}
    \Mat{I}_t & \Mat{0} & \Mat{\tH}^{-1}\Mat{C}_2 \\
    \Mat{0} & \Mat{I}_p & -\Mat{T}^{-1}\Mat{Q}_2 \\
    \Mat{0} & \Mat{0} & \Mat{I}_u \\
  \end{bmatrix} \begin{bmatrix}
    -\Mat{\tH}^{-1} & \Mat{0} & \Mat{0} \\
    \Mat{0} & \Mat{T}^{-1} & \Mat{0} \\
    \Mat{0} & \Mat{0} & \Mat{\tS}^{-1} \\
  \end{bmatrix} \begin{bmatrix}
    \Mat{I}_t & \Mat{0} & \Mat{0} \\
    \Mat{0} & \Mat{I}_p & \Mat{0} \\
    \Mat{C}_1\Mat{\tH}^{-1} & -\Mat{Q}_1\Mat{T}^{-1} & \Mat{I}_u \\
  \end{bmatrix}.
  \label{eq:P1}
\end{equation}

\begin{rem}
The multi-physics reduction approach proposed herein can be equivalently recast in other ways as well. Since we use a twofold approximation for $T^{-1}$, i.e., exact in $\mathcal{L}_1$, $\mathcal{U}_1$ and inexact in $\tS$, $\mathcal{M}_1^{-1}$ can be regarded as a member of the mixed constraint preconditioner class \cite{Bergamaschi2008,ferjangam08,Ferronato2010}. Similarly, the upper and lower block triangular factors in \eqref{eq:P1} play the role of decoupling operators for the original multi-physics problem and are the outcome of the general-purpose algebraic procedure defined in \cite{Ferronato2019}. Finally, $\mathcal{M}_1^{-1}$ can be also regarded as an example of application in a $3\times 3$ block non-symmetric context of the multigrid reduction framework \cite{Bui2018,Bui2020}, where fracture and body variables play the role of fine and coarse nodes, respectively, and $\tH$ replaces $H$ in matrix $\mathcal{J}$.
\end{rem}

Let us introduce the matrices:
\begin{subequations}
\begin{align}
    E_H &= I_t - \tH^{-1} H, \label{eq:EH} \\
    E_S &= I_u - \tS^{-1} S, \label{eq:ES}
\end{align}
\label{eq:E_tpu}\null
\end{subequations}
which can be regarded as a matrix measure of the quality of the approximations $\tH$ and $\tS$ introduced in $\mathcal{M}_1^{-1}$. The following result holds.

\begin{prop}
\label{th:eig_tpu}
The eigenvalues $\lambda_1 \in \mathbb{C}$ of the preconditioned matrix with the t-p-u approach are either 1, with multiplicity $n_p$, or such that:
\begin{equation}
    \left| \lambda_1 - 1 \right| \leq \left( 1 + \zeta \right) \max\{\varepsilon_H,\varepsilon_S\},
    \label{eq:eig_tpu}
\end{equation}
with $\varepsilon_H=\|E_H\|$, $\varepsilon_S=\|E_S\|$, $\zeta=\|\mathcal{Z}\|$ and
\begin{equation}
    \mathcal{Z} = \left[ \begin{array}{cc}
    \tH^{-1} C_2 \tS^{-1} C_1 & -\tH^{-1} C_2 \\ \tS^{-1} C_1 & 0 \end{array} \right]
    \label{eq:Z_prop1}
\end{equation}
for any compatible matrix norm.
\end{prop}

\begin{proof}
Recalling equations \eqref{eq:Q1} and \eqref{eq:P1} and introducing the error matrices \eqref{eq:E_tpu}, the preconditioned matrix with the t-p-u approach reads:
\begin{equation}
    \mathcal{M}_1^{-1} \mathcal{Q}_1 \mathcal{J} \mathcal{Q}_1^T = \left[ \begin{array}{ccc}
    I_t + \left( \tH^{-1} C_2 \tS^{-1} C_1 - I_t \right) E_H & 0 & -\tH^{-1} C_2 E_S \\ -T^{-1} Q_2 \tS^{-1} C_1 E_H & I_p & T^{-1} Q_2 E_S \\ \tS^{-1} C_1 E_H & 0 & I_u - E_S \end{array} \right],
    \label{eq:PJ_tpu}
\end{equation}
which has $n_p$ unitary eigenvalues. The remaining $n_t+n_u$ eigenvalues are those of the matrix $\mathcal{H}$ obtained by dropping the second block row and column from \eqref{eq:PJ_tpu}:
\begin{equation}
    \mathcal{H} = I_{t+u} + \left( \mathcal{Z} - I_{t+u} \right) \mathcal{E}_1,
    \label{eq:Kdef}
\end{equation}
with $I_{t+u}$ the identity matrix of order $n_t+n_u$ and $\mathcal{E}_1$:
\begin{equation}
    \mathcal{E}_1 = \left[ \begin{array}{cc}
    E_H & 0 \\ 0 & E_S \end{array} \right].
    \label{eq:Edef}
\end{equation}
The eigenvalues of $\mathcal{H}$ satisfy the bound \eqref{eq:eig_tpu}, thus closing the proof.
\end{proof}

\begin{rem}
Proposition \ref{th:eig_tpu} shows that the distance of $\tH^{-1}H$ from the identity and 
the approximation quality of $\tS^{-1}$ are key factors for the overall performance of $\mathcal{M}_1^{-1}$. 
While $\tH$ is fixed,
notice, however, that $\tS$ of equation \eqref{eq:appSchur_1} has algebraic properties that change with the fracture state and the evolution of the stick/slip/open regions throughout the active-set algorithm.  
In stick mode,
the contribution $\Mat{C}_1 \Mat{\tH}^{-1} \Mat{C}_2$ is symmetric positive semidefinite and $Q_2=0$. Hence, $\tS$ is SPD.  
Also in slip mode the contribution $\Mat{C}_1 \Mat{\tH}^{-1} \Mat{C}_2$ is
positive definite and $Q_2=0$, so $\tS$ remains positive definite, though slightly non-symmetric. 
With open elements, however, $\Mat{Q_2}$ depends on the current pressure solution and no theoretical considerations can be made in general. In these conditions,
$\tS$ is an indefinite non-symmetric matrix. 
\end{rem}

\subsection{Method no. 2: t-u-p approach}

An alternative multi-physics reduction sequence relies on the scheme sketched in the rightmost panel of Figure \ref{fig:MPReduction}. Introducing the traction variables \eqref{eq:t_reduc1} into the first equation of system \eqref{eq:ex_Jac_appl} yields:
\begin{equation}
    \left( A + C_1 \tH^{-1} C_2 \right) \mathbf{v}_u + Q_1 \mathbf{v}_p \simeq \mathbf{w}_u + C_1 \tH^{-1} \mathbf{w}_t,
    \label{eq:u_2}
\end{equation}
where the matrix:
\begin{equation}
    S_1 = A + C_1 \tH^{-1} C_2
    \label{eq:Schur1_2}
\end{equation}
is the first-level Schur complement. From a physical viewpoint, $S_1$ is an elasticity matrix with fictitious stiffness contributions arising along the fractures from the traction elimination. Then, a second reduction is needed by computing $\mathbf{v}_u$ from \eqref{eq:u_2} and introducing it in the third equation of \eqref{eq:ex_Jac_appl}:
\begin{equation}
    \left( T - Q_2 S_1^{-1} Q_1 \right) \mathbf{v}_p \simeq \mathbf{w}_p - Q_2 S_1^{-1} \left( \mathbf{w}_u + C_1 \tH^{-1} \mathbf{w}_t \right).
    \label{eq:p_2}
\end{equation}
The matrix at the left-hand side of equation \eqref{eq:p_2} is the second-level Schur complement:
\begin{equation}
    S_2 = T - Q_2 S_1^{-1} Q_1,
    \label{eq:Schur2_2}
\end{equation}
which, from a physical point of view, represents a modified transmissibility matrix including the effect of the stiffness of the 3D medium surrounding the fractures. Hence, the multi-physics reduction order is traction-displacement-pressure (t-u-p).

The computation of $S_1$ can be performed exactly, but its inverse has to be approximated. Since its nature is the same as that of $S$ of equation \eqref{eq:Schur_1}, we can effectively use an AMG operator, such as GAMG. We denote with $\tS_1^{-1}$ the operator that approximately applies $S_1^{-1}$. By distinction, $S_2$ cannot be computed exactly. Recalling the physical interpretation of $T$ and $S_1$, we can approximate the contribution $Q_2 S_1^{-1} Q_1$ by the diagonal fixed-stress matrix introduced as a preconditioner in \cite{whicastch16,caswhifer16}. Denoting by $S_K$ such a matrix, we approximate $S_2$ as:
\begin{equation}
    S_2 \simeq \tS_2 = T - S_K,
    \label{eq:app_Schur2_2}
\end{equation}
where the diagonal entries of $S_K$ are:
\begin{equation}
    [S_K]_k = \frac{|\Omega_k|}{\overline{K}_k}, \qquad k=1,\ldots,n_p,
    \label{eq:FS}
\end{equation}
with $|\Omega_k|$ a measure of the volume of the cells surrounding the $k$-th fracture element and $\overline{K}_k$ an estimate of the associated bulk modulus.

\begin{rem}
\label{rem:FS}
The computation of $|\Omega_k|$ and $\overline{K}_k$ in equation \eqref{eq:FS} can be carried out using the information arising from the discretization grid and the material properties. A more general algebraic strategy, however, can be implemented by following the ideas sketched in \cite{caswhifer16}. Recalling the definitions of $Q_1$ and $Q_2$, provided in \eqref{eq:Q1def} and \eqref{eq:Q2def}, respectively, it can be observed that the $k$-th column $Q_1^{(k)}$ of $Q_1$ and the $k$-th row $Q_2^{(k),T}$ of $Q_2$, $k=1,\ldots,n_p$, are sparse vectors with non-zero entries only in a small number of components corresponding to the indices of the degrees of freedom associated to some nodes of the 3D cells surrounding the $k$-th 2D fracture element. Let us denote with $\mathcal{K}_1^{(k)}$ and $\mathcal{K}_2^{(k)}$ the subsets of $\mathcal{K}=\{1,\ldots,n_u\}\subset\mathbb{N}$ containing the indices of the non-zero components of $Q_1^{(k)}$ and $Q_2^{(k),T}$, respectively. In general, $\mathcal{K}_1^{(k)}\neq\mathcal{K}_2^{(k)}$, with $\mathcal{K}_2^{(k)}$ possibly being the empty set. Setting $\mathcal{K}^{(k)}=\mathcal{K}_1^{(k)}\cup\mathcal{K}_2^{(k)}$, we define $R^{(k)}$ as the restriction operator from $\mathbb{R}^{n_u}$ to $\mathbb{R}^{|\mathcal{K}^{(k)}|}$ such that:
\begin{equation}
    S_1^{(k)} = R^{(k),T} S_1 R^{(k)}
    \label{eq:Ak}
\end{equation}
is the sub-matrix of $S_1$ made by the entries lying in the rows and columns with indices in $\mathcal{K}^{(k)}$. Since $S_1^{(k)}$ is a diagonal block of an SPD matrix, it is non-singular and can be regularly inverted. The $k$-th diagonal entry of $S_K$ can be therefore computed as:
\begin{equation}
    [S_K]_k = Q_2^{(k),T} S_1^{(k),-1} Q_1^{(k)}, \qquad k=1,\ldots,n_p,
    \label{eq:FS_algebraic}
\end{equation}
which provides a fully algebraic interpretation of the classical fixed-stress contribution \eqref{eq:FS}.
\end{rem}

The construction and application of the resulting preconditioning operator $\mathcal{M}_2^{-1}$ with the t-u-p approach is summarized in Algorithms \ref{alg:cpt_tup} and \ref{alg:app_tup}. The FS($Q_2,S_1,Q_1$) operator gives the diagonal matrix $S_K$ with entries computed as in equation \eqref{eq:FS_algebraic}. Since $n_u$ is generally significantly larger than $n_t$ and $n_p$, the cost for applying the exact inverse of $\tH$ and $\tS_2$ is negligible with respect to the AMG algorithm for $\tS_1^{-1}$. Therefore, the cost per iteration for the $\mathcal{M}_2^{-1}$ application is roughly two AMG calls, i.e., twice that of $\mathcal{M}_1^{-1}$.

\begin{algorithm}
\caption{ {\sc Preconditioner Construction: t-u-p approach}
[$\tH$, $S_1$, $\tS_2$]=cpt\_tup($\mathcal{J}$).}
\label{alg:cpt_tup}
\begin{algorithmic}[1]
\State $\tH = \text{Bdiag}(H,3)$;
\State $S_1 = A + C_1 \tH^{-1} C_2$;
\State $\tS_2 = T - \text{FS}(Q_2,S_1,Q_1)$;
\end{algorithmic}
\end{algorithm}

\begin{algorithm}
\caption{ {\sc Preconditioner Application: t-u-p approach}
[$\mathbf{v}_u$, $\mathbf{v}_t$, $\mathbf{v}_p$]=app\_tup($\mathcal{J}$, $\tH$, $S_1$, $\tS_2$, $\mathbf{w}_u$, $\mathbf{w}_t$, $\mathbf{w}_p$).}
\label{alg:app_tup}
\begin{algorithmic}[1]
\State $\mathbf{t}_t = \tH^{-1} \mathbf{w}_t$;
\State $\mathbf{t}_u = \mathbf{w}_u + C_1 \mathbf{t}_t$;
\State $\mathbf{s}_u = \text{AMG}(S_1,\mathbf{t}_u)$;
\State $\mathbf{t}_p = \mathbf{w}_p - Q_2 \mathbf{s}_u$;
\State $\mathbf{v}_p = \tS_2^{-1} \mathbf{t}_p$;
\State $\mathbf{t}_u = Q_1 \mathbf{v}_p$;
\State $\mathbf{v}_u = \text{AMG}(S_1, \mathbf{t}_u)$;
\State $\mathbf{v}_u \gets \mathbf{s}_u - \mathbf{v}_u$;
\State $\mathbf{s}_t = C_2 \mathbf{v}_u$;
\State $\mathbf{v}_t = -\mathbf{t}_t + \tH^{-1} \mathbf{s}_t$;
\end{algorithmic}
\end{algorithm}

Similarly to the t-p-u approach, the preconditioning operator $\mathcal{M}_2^{-1}$ can be written as an inexact block LDU factorization of $\mathcal{J}$. With
the permutation matrix $\mathcal{Q}_2$:
\begin{equation}
  \mathcal{Q}_2 = \begin{bmatrix}
    \Mat{0} & \Mat{I}_t & \Mat{0} \\
    \Mat{I}_u & \Mat{0} & \Mat{0} \\
    \Mat{0} & \Mat{0} & \Mat{I}_p \\
  \end{bmatrix},
  \label{eq:Q2}
\end{equation}
the block LDU factorization reads:
\begin{equation}
  \mathcal{Q}_2 \mathcal{J} \mathcal{Q}_2^T \simeq \mathcal{L}_2 \mathcal{D}_2 \mathcal{U}_2,
\end{equation}
with:
\begin{align}
  \mathcal{L}_2 &= \begin{bmatrix}
    \Mat{I}_t & \Mat{0} & \Mat{0} \\
    -\Mat{C}_1\Mat{\tH}^{-1} & \Mat{I}_u & \Mat{0} \\
    \Mat{0} & \Mat{Q}_2\Mat{\tS}_1^{-1} & \Mat{I}_p \\
  \end{bmatrix}, &&&
  \mathcal{D}_2 &= \begin{bmatrix}
    -\Mat{\tH} & \Mat{0} & \Mat{0} \\
    \Mat{0} & \Mat{S}_1 & \Mat{0} \\
    \Mat{0} & \Mat{0} & \Mat{\tS}_2 \\
  \end{bmatrix} & \text{and} &&
  \mathcal{U}_2 &= \begin{bmatrix}
    \Mat{I}_t & -\Mat{\tH}^{-1}\Mat{C}_2 & \Mat{0} \\
    \Mat{0} & \Mat{I}_u & \Mat{\tS}_1^{-1}\Mat{Q}_1 \\
    \Mat{0} & \Mat{0} & \Mat{I}_p \\
  \end{bmatrix}.
  \label{eq:LDU2}
\end{align}
The final algebraic expression of $\mathcal{M}_2^{-1}$ is therefore:
\begin{equation}
    \mathcal{M}_2^{-1} = \left[ \begin{array}{ccc}
    I_t & \tH^{-1}C_2 & -\tH^{-1} C_2 \tS_1^{-1} Q_1 \\
    0 & I_u & -\tS_1^{-1} Q_1 \\
    0 & 0 & I_p \end{array} \right]
    \left[ \begin{array}{ccc}
    -\tH^{-1} & 0 & 0 \\
    0 & \tS_1^{-1} & 0 \\
    0 & 0 & \tS_2^{-1} \end{array} \right]
    \left[ \begin{array}{ccc}
    I_t & 0 & 0 \\
    C_1 \tH^{-1} & I_u & 0 \\
    -Q_2 \tS_1^{-1} C_1 \tH^{-1} & -Q_2 \tS_1^{-1} & I_p \end{array} \right].
    \label{eq:P2}
\end{equation}

\begin{rem}
Like the t-p-u approach, also $\mathcal{M}_2^{-1}$ can be equivalently recast in other ways. For instance, it can be viewed again as a mixed constraint preconditioner applied to the matrix:
\begin{equation}
    \hat{\mathcal{J}} = \left[ \begin{array}{cc}
    \mathcal{A} & \mathcal{B}_1 \\
    \mathcal{B}_2 & T \end{array} \right],
    \qquad \mbox{ with } \; \mathcal{A} = \left[ \begin{array}{cc} 
    -H & C_2 \\ C_1 & A \end{array} \right], \quad \mathcal{B}_1 = \left[ \begin{array}{c}
    0 \\ Q_1 \end{array} \right], \quad \mathcal{B}_2 = \left[ \begin{array}{cc} 0 & Q_2 \end{array} \right],
    \label{eq:22_block}
\end{equation}
where the inverse of the leading block $\mathcal{A}$ is approximated by an inner inexact constraint preconditioner \cite{Bergamaschi2007,Janna2012}. Alternatively, $\mathcal{M}_2^{-1}$ can be regarded as a double application of the multigrid reduction framework: (i) to matrix $\mathcal{A}$, with $\tH$ instead of $H$,
and (ii) to matrix $\mathcal{S}$:
\begin{equation}
    \mathcal{S} = \left[ \begin{array}{cc}
    S_1 & Q_1 \\
    Q_2 & T \end{array} \right].
    \label{eq:RP_tup2}
\end{equation}
\end{rem}

Let us introduce the matrices:
\begin{subequations}
\begin{align}
    E_{S_1} &= I_u - \tS_1^{-1} S_1, \label{eq:ES1} \\
    E_{S_2} &= I_p - \tS_2^{-1} \hat{S}_2, \label{eq:ES2}
\end{align}
\label{eq:E2def}\null
\end{subequations}
where $\hat{S}_2=T-Q_2\tS_1^{-1}Q_1$. Along with $E_H$ already introduced in equation \eqref{eq:EH}, $E_{S_1}$ and $E_{S_2}$ are matrix measures of the quality of the approximations $\tS_1$ and $\tS_2$ introduced in $\mathcal{M}_2^{-1}$. The following result holds.

\begin{prop}
\label{thm:eig_tup}
The eigenvalues $\lambda_2\in\mathbb{C}$ of the preconditioned matrix with the t-u-p approach are such that:
\begin{equation}
    |\lambda_2-1| \leq (1+\gamma) \max\{\varepsilon_H,\varepsilon_{S_1}, \varepsilon_{S_2}\},
    \label{eq:eigbound2}
\end{equation}
with $\varepsilon_H=\|E_H\|$, $\varepsilon_{S_1}=\|E_{S_1}\|$, $\varepsilon_{S_2}=\|E_{S_2}\|$, $\gamma=\|\mathcal{G}\|$ and
\begin{equation}
    \mathcal{G} = \left[ \begin{array}{ccc}
    C_{2H}\hat{I}_QC_{1S} & -C_{2H}\hat{I}_Q & C_{2H}Q_{1S} \\
    \hat{I}_QC_{1S} & -Q_{1S}Q_{2S} & Q_{1S} \\
    -Q_{2S}C_{1S} & Q_{2S} & 0 \end{array} \right], 
    \label{eq:Gdef}
\end{equation}
\begin{equation}
    C_{1S} = \tS_1^{-1}C_1, \quad C_{2H} = \tH^{-1}C_2, \quad Q_{1S} = \tS_1^{-1}Q_1, \quad Q_{2S} = \tS_2^{-1}Q_2, \quad \hat{I}_Q = I_u + Q_{1S}Q_{2S},
    \label{eq:matdef}
\end{equation}
for any compatible matrix norm.
\end{prop}

\begin{proof}
Recalling equations \eqref{eq:Q2} and \eqref{eq:P2} and introducing the error matrices \eqref{eq:EH} and \eqref{eq:E2def}, the preconditioned matrix with the t-u-p approach reads:
\small
\begin{equation}
    \mathcal{M}_2^{-1} \mathcal{Q}_2 \mathcal{J} \mathcal{Q}_2^T = \left[ \begin{array}{ccc}
    I_t + \left[ -I_t+\tH^{-1}C_2 \left( I_u + \tS_1^{-1}Q_1\tS_2^{-1}Q_2 \right) \tS_1^{-1}C_1 \right] E_H & -\tH^{-1}C_2 \left( I_u + \tS_1^{-1}Q_1\tS_2^{-1}Q_2 \right) E_{S_1} & \tH^{-1}C_2\tS_1^{-1}Q_1 E_{S_2} \\
    \left( I_u + \tS_1^{-1}Q_1\tS_2^{-1}Q_2 \right) \tS_1^{-1}C_1 E_H & I_u - \left( I_u + \tS_1^{-1}Q_1\tS_2^{-1}Q_2 \right) E_{S_1} & \tS_1^{-1}Q_1 E_{S_2} \\
    -\tS_2^{-1}Q_2\tS_1^{-1}C_1 E_H & \tS_2^{-1}Q_2 E_{S_1} & I_p-E_{S_2} \end{array} \right].
    \label{eq:precmat_tup}
\end{equation}
\normalsize
By making use of the definitions \eqref{eq:Gdef} and \eqref{eq:matdef}, equation \eqref{eq:precmat_tup} can be re-written as:
\begin{equation}
    \mathcal{M}_2^{-1} \mathcal{Q}_2 \mathcal{J} \mathcal{Q}_2^T = \mathcal{I} + \left( \mathcal{G} - \mathcal{I} \right) \mathcal{E}_2,
    \label{eq:precmat_tup2}
\end{equation}
with:
\begin{equation}
    \mathcal{I} = \left[ \begin{array}{ccc}
    I_t & 0 & 0 \\ 0 & I_u & 0 \\ 0 & 0 & I_p \end{array} \right], \qquad \mathcal{E}_2 = \left[ \begin{array}{ccc}
    E_H & 0 & 0 \\ 0 & E_{S_1} & 0 \\ 0 & 0 & E_{S_2} \end{array} \right].
    \label{eq:I_E2}
\end{equation}
The eigenvalues $\lambda_2$ of the matrix in equation \eqref{eq:precmat_tup2} satisfy the bound \eqref{eq:eigbound2}, thus closing the proof.
\end{proof}

\begin{rem}
The outcome of Proposition \ref{thm:eig_tup} is very close to that of Proposition \ref{th:eig_tpu}, so the two approaches are expected to behave similarly. However, with $\mathcal{M}_2^{-1}$ an additional contribution to the error arises, $\varepsilon_{S_2}$, and there is no guarantee that a cluster of eigenvalues is exactly 1. Moreover, as already observed, the cost for the $\mathcal{M}_2^{-1}$ application is approximately twice that of $\mathcal{M}_1^{-1}$. To decrease this cost of $\mathcal{M}_2^{-1}$, one can apply an incomplete block factorization, i.e., only $\mathcal{D}_2^{-1} \mathcal{L}_2^{-1}$. The preconditioner $\mathcal{M}_2^{\star,-1}$ could behave similarly to the original block preconditioner $\mathcal{M}_2^{-1}$, but almost halving the application cost.
\end{rem}

\begin{rem}
As with the t-p-u approach, the Schur complements $S_1$ and $\tS_2$, equation \eqref{eq:Schur1_2} and \eqref{eq:app_Schur2_2}, respectively, have algebraic properties that change with the fracture state and the evolution of the stick/slip/open partitioning. $S_1$ is symmetric in stick mode and non-symmetric otherwise, but in any case is positive definite. By distinction, $\tS_2$ is symmetric anyway, but can be indefinite with open fracture elements, because $Q_2$ depends on the current pressure solution and no a-priori considerations can be done.
\end{rem}

Recalling equation \eqref{eq:Q2vsQ1}, we have that the Schur complement $S$ of the t-p-u approach is linked to $S_1$ by:
\begin{equation}
  \Mat{S} = \Mat{S}_1 + \frac{1}{\Delta t}\Mat{Q}_1 \Mat{\tT}^{-1} \Psi\left(\Mat{Q}_1^T\right) - \Mat{Q}_1
    \Mat{\tT}^{-1} \Mat{F}_u.
\end{equation}
The possible indefiniteness of $S$ is related to the contribution depending on matrix $\Mat{F}_u$. The second Schur complement $S_2$ reads:
\begin{equation}
  \Mat{S}_2 = \Mat{T} + \frac{1}{\Delta t}\Psi\left(\Mat{Q}_1^T\right) \Mat{\tS}_1^{-1} \Mat{Q}_1 - \Mat{F}_u
    \Mat{S}_1^{-1} \Mat{Q}_1.
\end{equation}
Again, the sum of the first two contributions is positive definite, while the indefiniteness arises from the term depending on $\Mat{F}_u$.
There is no general indication on the actual probability of either $S$ or $S_2$ to become indefinite during a full simulation.

\section{Numerical results}
\label{sec:numres}

Three sets of numerical experiments are used to investigate the robustness, scalability and computational performance of the proposed preconditioning framework. The first set (Test 1) consists of a small size single-fracture problem and is used to analyze and compare the robustness of the t-p-u and t-u-p approaches. The second set (Test 2) simulates the behavior of a number of uniformly discretized fractures with the aim at investigating the weak scalability. Finally, we consider a large-size realistic application (Test 3), representing a tilted well in a hydraulic fracturing stimulation process, in order to verify the computational efficiency in a meaningful context.

\begin{figure}
    \centering
    \includegraphics[width=8cm]{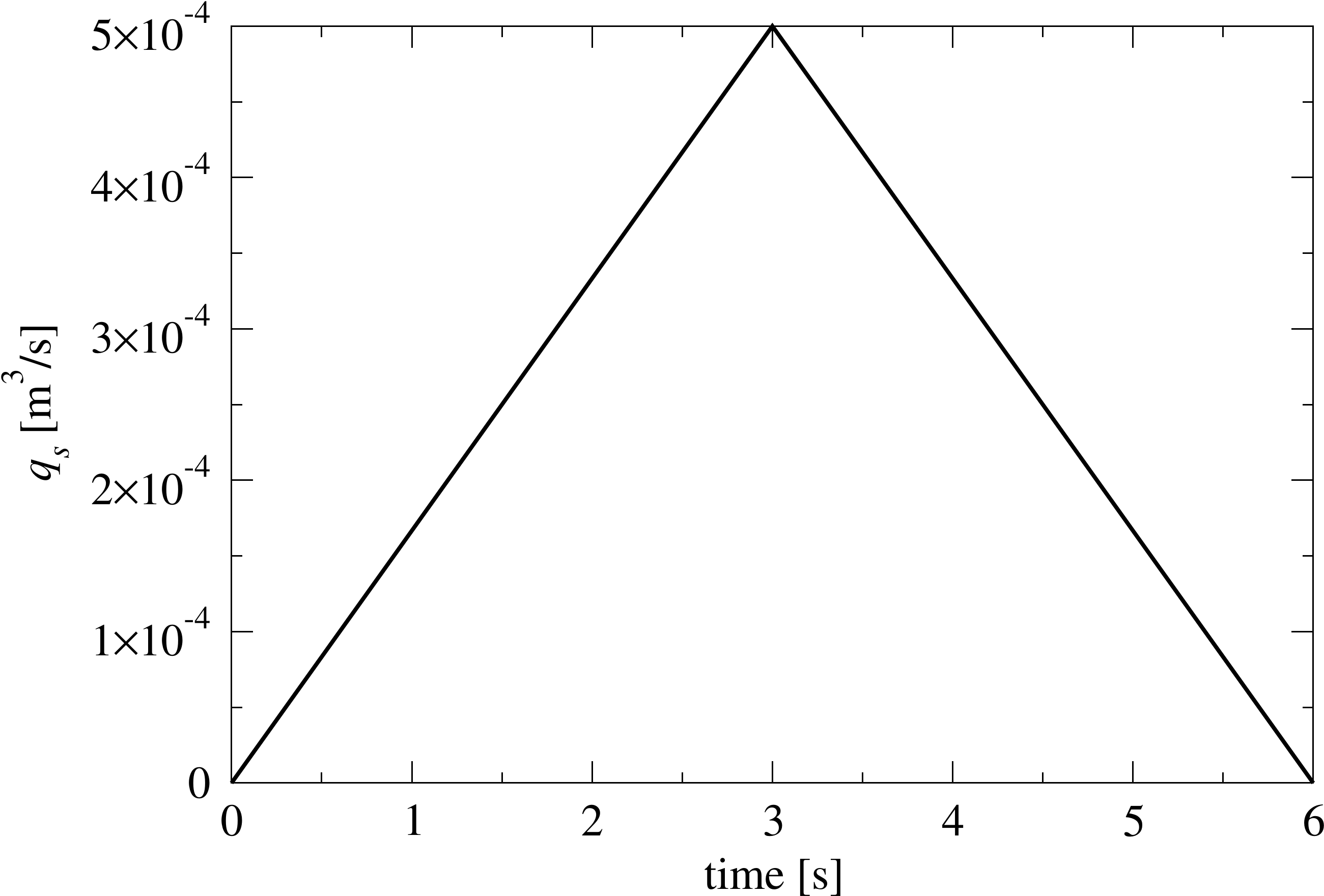}
    \caption{Flow rate in time for all the test cases.}
    \label{fig:q_history}
\end{figure}

In all test cases, a linearly increasing flow discharge $q_s$ is first injected in the fractures, with a maximum value of $5 \cdot 10^{-4}$ m$^3$s$^{-1}$, and then extracted, following the time history depicted in Figure \ref{fig:q_history}. The overall process covers \mbox{$t_{\max}=6$ s} and is discretized into 12 uniform time steps ($\Delta t=0.5$ s). The homogeneous rock material has Young's modulus and Poisson's ratio equal to $3\cdot 10^3$ MPa and $0.25$, respectively. The friction
coefficient for the fractures is $0.577$, i.e., the friction angle $\theta$ is $30^\circ$, with the
cohesion $c$ set to zero. According to \cite{kamenov2013laboratory}, the initial
conductivity value $C_{f,0}$ of equation \eqref{eq:cond_def} is equal to $10 \
\text{mD}\cdot\text{m} = 9.87 \cdot 10^{-15} \ \text{m}^2 \cdot \text{m}$. On the corners of the fracture boundary
$\partial\Gamma_i$ embedded in the 3D body $\Omega$, a constant zero pressure ($p_0 = 0$ MPa) is imposed.

For each time step, an outer loop (active-set strategy) is
coupled with an inner loop (Newton's method). Starting from the last converged stick/slip/open region partitioning, we iterate with Newton's method until the 2-norm of the non-linear residual \eqref{eq:discrete_res} is reduced by a factor $10^{-6}$. At convergence, a consistency check is carried out to verify the active/inactive status of each fracture element. If the initial stick/slip/open region partitioning has to be updated, another inner Newton's loop is performed. At each Newton's iteration, the system \eqref{eq:JacSys} is solved by a right-preconditioned full GMRES algorithm \cite{saasch86}, initialized by the zero vector and stopped when the linear residual is reduced by a factor $10^{-8}$. Hence, for each entire simulation we consider the total number of iterations needed for the outer active-set strategy, $N_{\ell}$, the inner Newton's loop, $N_N$, and the preconditioned GMRES method, $N_G$. The average GMRES iteration count for a single linear system is denoted by $\overline{N}_G$, with $N_{G,\min}$ and $N_{G,\max}$ the minimum and maximum value, respectively, required during the entire simulation.

\subsection{Test 1: Robustness}

The purpose of this test case is to verify the robustness of the proposed t-p-u and t-u-p approaches and highlight the main differences. A relatively small-size problem is set up, with a single 2-m long vertical fracture completely passing through the 3D domain, which is a box with sizes $10 \times 10 \times 0.15$ m (Figure \ref{fig:EI_Mesh}). 
The
top and bottom surfaces are fixed along the $z$-direction, while the vertical sides approximately parallel to the fracture
are compressed by a uniform normal load $\sigma_0 =
100\ \text{MPa}$. The out-of-plane displacement of the two remaining external surfaces of
the 3D body is prevented. 
The
computational grid consists of 2,944 nodes, 2,046 elements and 120 fracture elements, with the following number of
unknowns: $n_u = 8,832$, $n_t = 360$, and $n_p = 120$.

\begin{figure}
  \centering
  \null\hfill
  \includegraphics[width=0.6\linewidth]{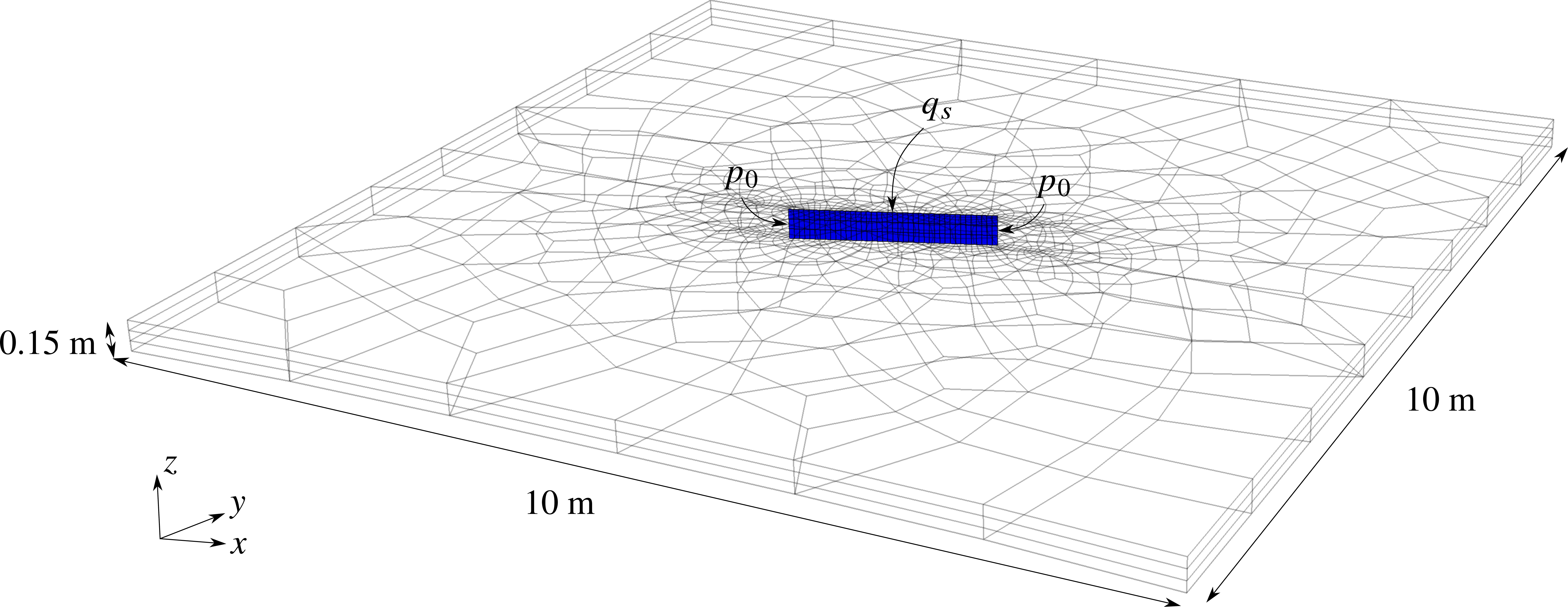}
  \hfill\null
  \caption{Test 1: Domain configuration, computational grid and injection location. The vertical exaggeration
    factor is 2.}
  \label{fig:EI_Mesh}
\end{figure}

\begin{table}
  \centering
  \begin{tabular}{c|cccccccccccc|cc}
   $t$ [s] & 0.5 & 1.0 & 1.5 & 2.0 & 2.5 & 3.0 & 3.5 & 4.0 & 4.5 & 5.0 & 5.5 & 6.0 & $N_{\ell}$ & $N_N$ \\
  \hline
   t-p-u   & 2.0 & 5.0 & 4.0 & 4.6 & 7.8 & 9.3 & 5.9 & 5.4 & 4.9 & 4.6 & 4.5 & 2.0 & 58 & 253 \\
   t-u-p   & 2.0 & 5.0 & 4.0 & 4.0 & 6.5 & 7.4 & 7.2 & 5.0 & 2.0 & 2.0 & 2.0 & 2.0 & 33 & 153 \\
   \hline
   t-p-u   & 80.0 & 72.0 & 72.0 & 62.2 & 57.3 & 43.0 &   -- &   -- &   -- &   -- &   -- &   -- & 22$^*$ & 79$^*$ \\
   t-u-p   & 80.0 & 72.0 & 72.0 & 61.8 & 56.1 & 30.9 & 16.0 & 27.7 & 46.2 & 67.6 & 81.0 & 81.0 & 43 & 199 \\
  \end{tabular}
  \caption{Test 1: Average number of linear iterations $\overline{N}_G$ for the t-p-u and t-u-p approach with nested direct solvers (upper rows) and AMG algorithms (lower rows). The total number of active-set and Newton iterations, $N_{\ell}$ and $N_N$, are also given. * means that $N_{\ell}$ and $N_N$ refer to the first six steps only.}
  \label{tab:cfr_iters}
\end{table}

We start the analysis by using nested direct methods to apply $\tS^{-1}$ in the t-p-u approach and $\tS_1^{-1}$ in the t-u-p approach. The performance obtained with this configuration can be viewed as the best potential outcome with regard to the average number of linear iterations $\overline{N}_G$. The result is reported in the upper rows of Table \ref{tab:cfr_iters}.
The two approaches behave quite similarly as far as $\overline{N}_G$ 
is concerned. 
The very small average number of GMRES iterations to converge confirms that the approximations introduced in the construction of $\tH$, $\tS$ and $\tS_2$ are pretty much acceptable.
Recall, however, that the t-u-p application cost per iteration is roughly twice that of t-p-u, so that the first approach appears to guarantee an overall better computational performance.

This result can significantly change when AMG solvers are introduced instead of inner direct methods, as it is in practice mandatory for large-size real-world simulations. The values of $\overline{N}_G$ for each time step, $N_{\ell}$ and $N_N$ are provided in the lower rows of Table \ref{tab:cfr_iters}.
It can be noticed that again the reported values are
almost the same for the two approaches in the first six simulation steps. 
Notice that the significant increase of $\overline{N}_G$ with respect to the use of nested direct solvers is mainly due to the bad elemental aspect ratio, which is as small as $\sim\! 10^{-2}$,
used to create this test case. Such a distorted grid negatively affects the conditioning of $\tS$ and $S_1$, which the AMG inner preconditioner is not able to fully address.
%
%
The main difference is met at the seventh time step, where GMRES convergence cannot be achieved with the t-p-u approach.
Although the t-p-u reduction method is generally more efficient than t-u-p, it appears to be less robust in some configurations of the stick/slip/open fracture regions.
Similar
behaviors have also been observed with other numerical experiments as well. 

The reason for such an outcome stems from the algebraic properties of the (approximate) Schur complement matrices, $\tS$ in the t-p-u approach and $S_1$-$\tS_2$ in the t-u-p approach, arising at the seventh time step of the simulation. To this aim, we analyze the eigenvalue distributions of such matrices (Figure \ref{fig:EI_eig}).
%
With the current stick/slip/open region configuration, $\tS$ turns out to be non-symmetric and one eigenvalue with a negative real part arises. Notice also the large ratio between the maximum and minimum eigenvalue modulus (Table \ref{tab:EI_eig}). In this condition, the AMG method used to approximate the application of $\tS^{-1}$ loses its theoretical properties and is no longer effective. Replacing AMG with another indefinite inner preconditioner for $\tS$, such as an incomplete LU factorization with partial fill-in, can fix this issue, at the cost of losing the method scalability.
By distinction, with the t-u-p approach we have the theoretical guarantee that the first-level Schur complement $S_1$ is positive definite independently on the fracture condition. As it can be seen from Figure \ref{fig:EI_eig} and Table \ref{tab:EI_eig}, here AMG appears to work quite effectively despite the conditioning of $S_1$ is the same as $\tS$. The second-level Schur complement $\tS_2$ is symmetric by construction, but can be indefinite. In this case, $\tS_2$ turns out to be still positive definite, but in any case its inversion by a nested direct solver would ensure an effective preconditioner behavior.  

\begin{figure}
  \centering
  \includegraphics[width=1\linewidth]{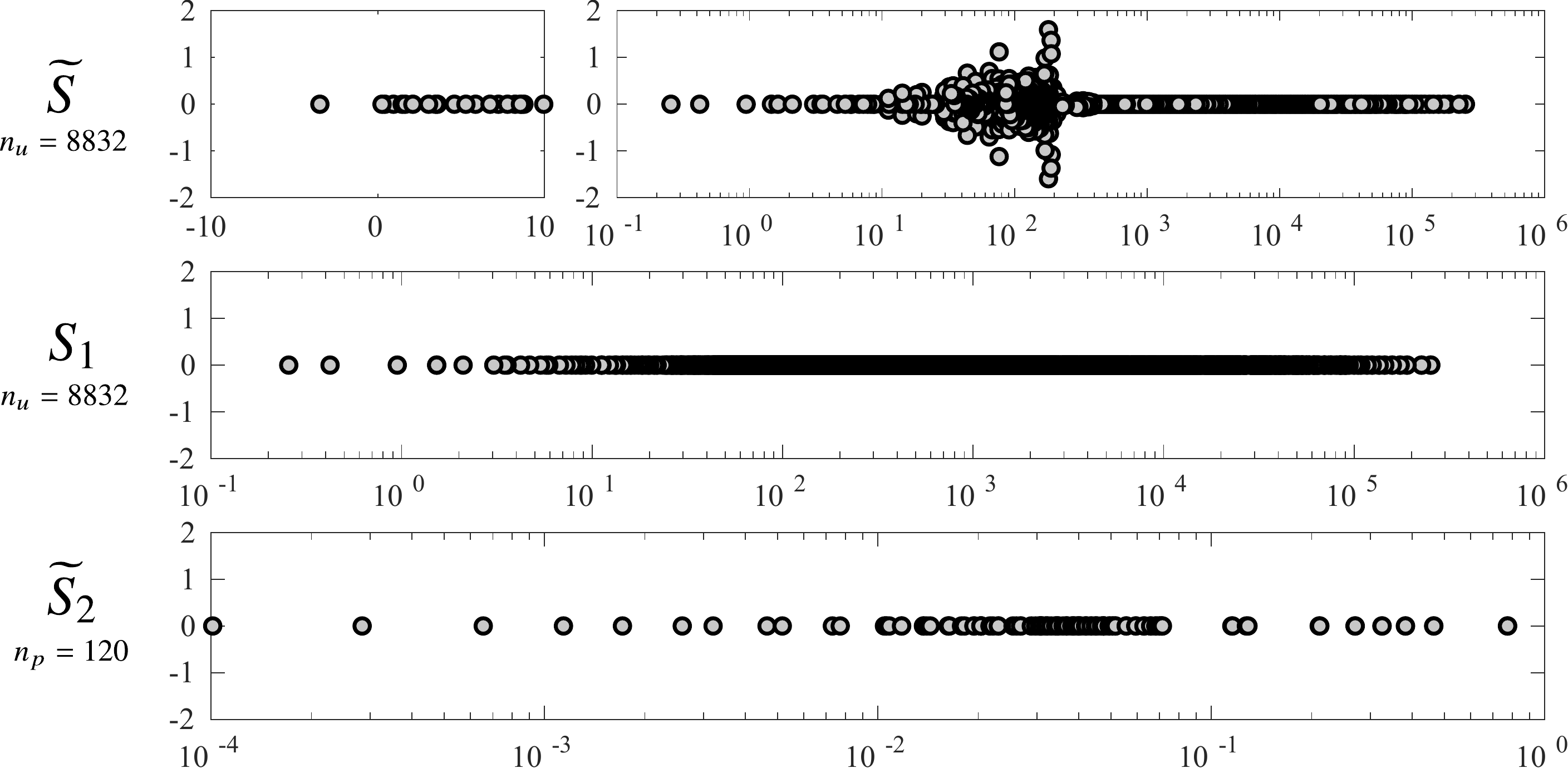}
  \caption{Test 1: Eigenvalue distribution in the complex plane for $\tS$, $S_1$ and $\tS_2$ at the seventh time step. 
  For $\tS$ a zoom around zero in arithmetic scale is also reported.}
  \label{fig:EI_eig}
\end{figure}

\begin{table}
  \centering
  \begin{tabular}{c|cc|c}
   & $|\lambda_{\max}|$ & $|\lambda_{\min}|$ & $|\lambda_{\max}|/|\lambda_{\min}|$ \\
  \hline
   $\tS$   & $2.52 \cdot 10^5$    & $2.56 \cdot 10^{-1}$ & $9.88 \cdot 10^5$ \\
   $S_1$   & $2.52 \cdot 10^5$    & $2.56 \cdot 10^{-1}$ & $9.88 \cdot 10^5$ \\
   $\tS_2$ & $7.74 \cdot 10^{-1}$ & $1.01 \cdot 10^{-4}$ & $7.66 \cdot 10^3$ \\
  \end{tabular}
  \caption{Test 1: Maximum and minimum eigenvalue modulus for $\tS$, $S_1$ and $\tS_2$ at the seventh time step.}
  \label{tab:EI_eig}
\end{table}


On summary, we can conclude that the t-p-u approach can be computationally more efficient than t-u-p. However, it may lack of robustness in large problems whenever the stick/slip/open region partitioning in a full simulation gives rise to an indefinite matrix $\tS$. The t-u-p approach, though more expensive, is also much more robust, because it concentrates the source of the possible numerical issues into $\tS_2$, which is a symmetric matrix with a 2D graph connection that can be effectively addressed by a nested direct solver. For this reason, the t-u-p approach is to be preferred in a full simulation, where unpredictable fracture configurations may arise. Therefore, in the next numerical experiments we will focus on t-u-p approach alone.

\subsection{Test 2: Weak scalability}

To investigate the weak scalability of the proposed algorithm, we first consider the test case shown in Figure
\ref{fig:WS_MeshPressure} (Test 2a), consisting of a unitary cube with four vertical fractures. Fluid injection and extraction is prescribed at the center of each fracture, simulating the action of a horizontal well.
The external faces parallel to the fractures are subjected to a
compressive constant load ($\sigma_0 = 10$ MPa), while the displacement on the other boundary faces is prevented. 
The pressure solution on the fracture surfaces at time $t=3$ s, i.e., at the maximum fluid injection rate, is also shown in Figure \ref{fig:WS_MeshPressure}.

\begin{figure}
  \centering
  \null\hfill
  \includegraphics[height=0.3\linewidth]{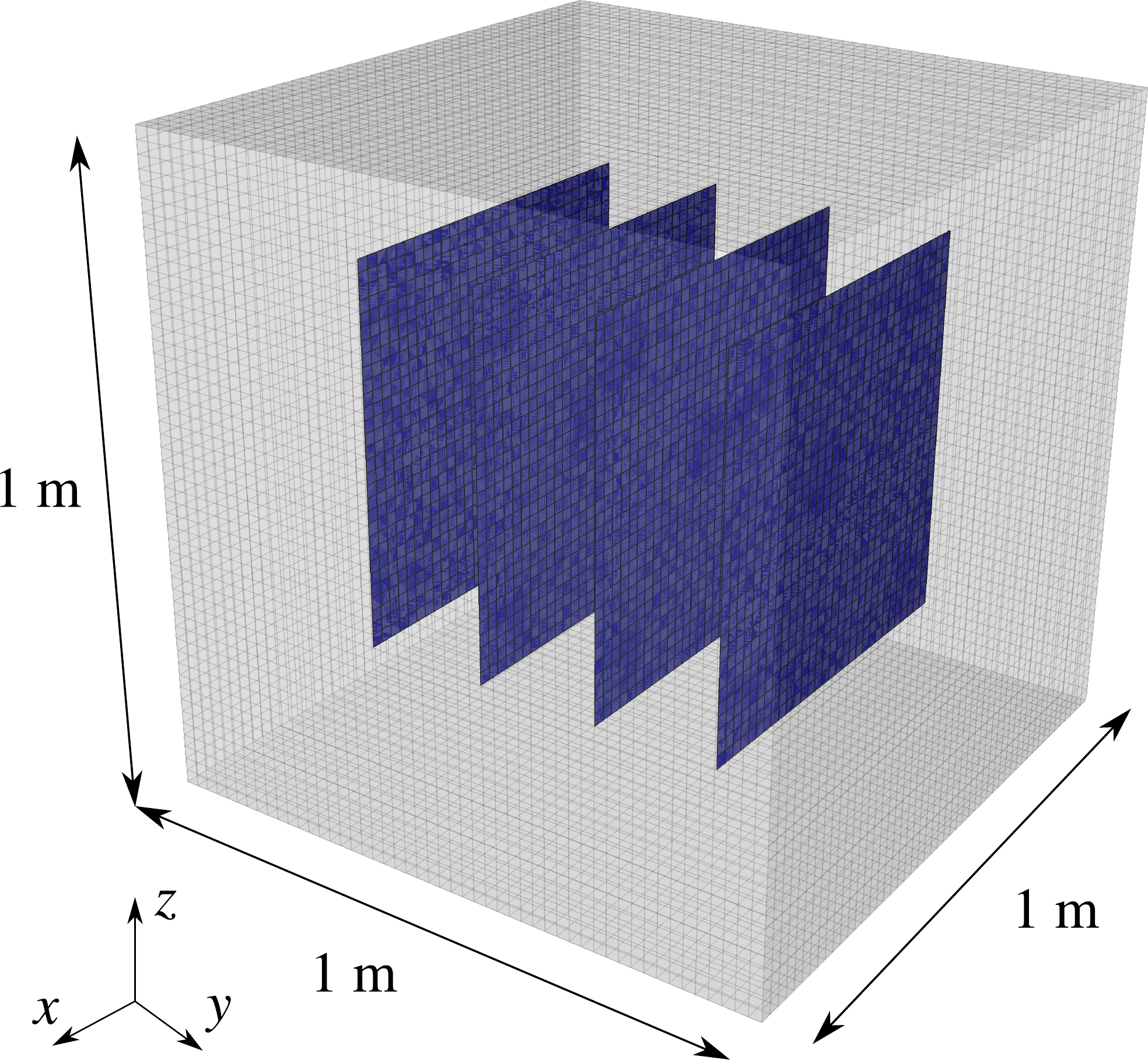}\hfill
  \includegraphics[height=0.3\linewidth]{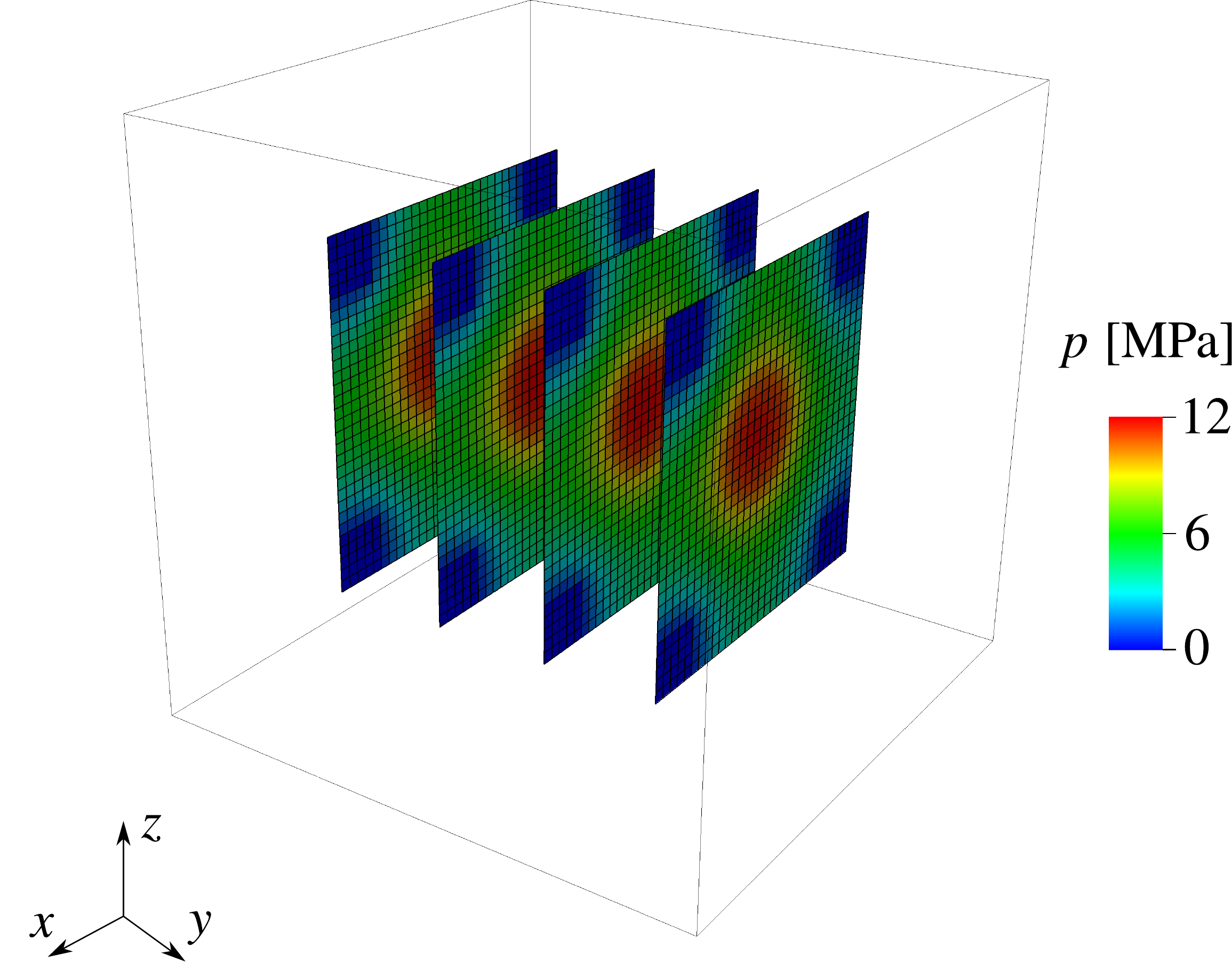}
  \hfill\null
  \caption{Test 2a: Domain configuration (left) and pressure solution at time $t=3$ s (right).}
  \label{fig:WS_MeshPressure}
\end{figure}


\begin{table}
  \centering
  \begin{tabular}{c|cc}
          & coarse &   fine \\
    \hline
    $n_u$ &  55,080 & 408,045 \\
    $n_t$ &   2,700 &  10,800 \\
    $n_p$ &    900 &   3,600 \\
    \hline
    total &  58,680 & 422,445 \\
  \end{tabular}
  \caption{Test 2a: Problem size for the coarse and fine mesh.}
  \label{tab:test2a}
\end{table}

Two uniform grid refinements are considered, with the linear elemental size varying from $h = 0.04 \ \text{m}$ to $h = 0.02 \ \text{m}$. The number of unknowns for the coarse and fine discretization is listed in Table \ref{tab:test2a}.
The overall performance of the non-linear solution algorithm is reported in Table \ref{tab:t2_nnlin}, where the number of: (i) steps of the active-set algorithm, and (ii) iterations of the inner Newton's loop, are reported for every simulation time
for both the coarse and fine discretization.
The top panel of Figure \ref{fig:t2_convNonLin} graphically summarizes the same pieces of
information.

\begin{table}
  \centering
  \begin{tabular}{lc|c|c|c|c|c|c|c|c|c|c|c|c|}
    & & \multicolumn{12}{|c|}{time [s]} \\
    \hline
    & $\ell$ & 0.5 & 1.0 & 1.5 & 2.0 & 2.5 & 3.0 & 3.5 & 4.0 & 4.5 & 5.0 & 5.5 & 6.0 \\
    \hline
    & 1 & 2 &
      2 & 2 & 2 & 5 & 6 & 6 & 6 & 4 & 2 & 2 & 2 \\
    coarse & 2 &   &   &   & 5 & 4 & 5 & 5 & 5 & 2 &   &   &   \\
    & 3 &   &   &   &   & 3 &   & 5 &   &   &   &   &   \\
%
    \hline
    & 1 & 2 &
      2 & 2 & 6 & 6 & 5 & 7*& 6*& 6 & 3 & 2 & 2 \\
    & 2 &   &   & 5 & 5 & 5 & 6 & 7 & 6 & 5 & 2 &   &   \\
    fine & 3 &   &   & 4 & 4 & 4 & 5 & 5 & 5 & 5 &   &   &   \\
    & 4 &   &   & 4 & 3 &   & 4 & 5 & 4 & 4 &   &   &   \\
    & 5 &   &   &   &   &   &   & 4 & 4 &   &   &   &   \\
  \end{tabular}
  \caption{Test 2a: Newton's iterations $N_N$ for each simulation time and active-set step $\ell$ for the coarse (top rows) and the fine mesh (bottom rows). * means that Newton's scheme
    does not converge and starting from a completely closed configuration is required.}
  \label{tab:t2_nnlin}
\end{table}


\begin{figure}
  \centering
  \includegraphics[width=\linewidth]{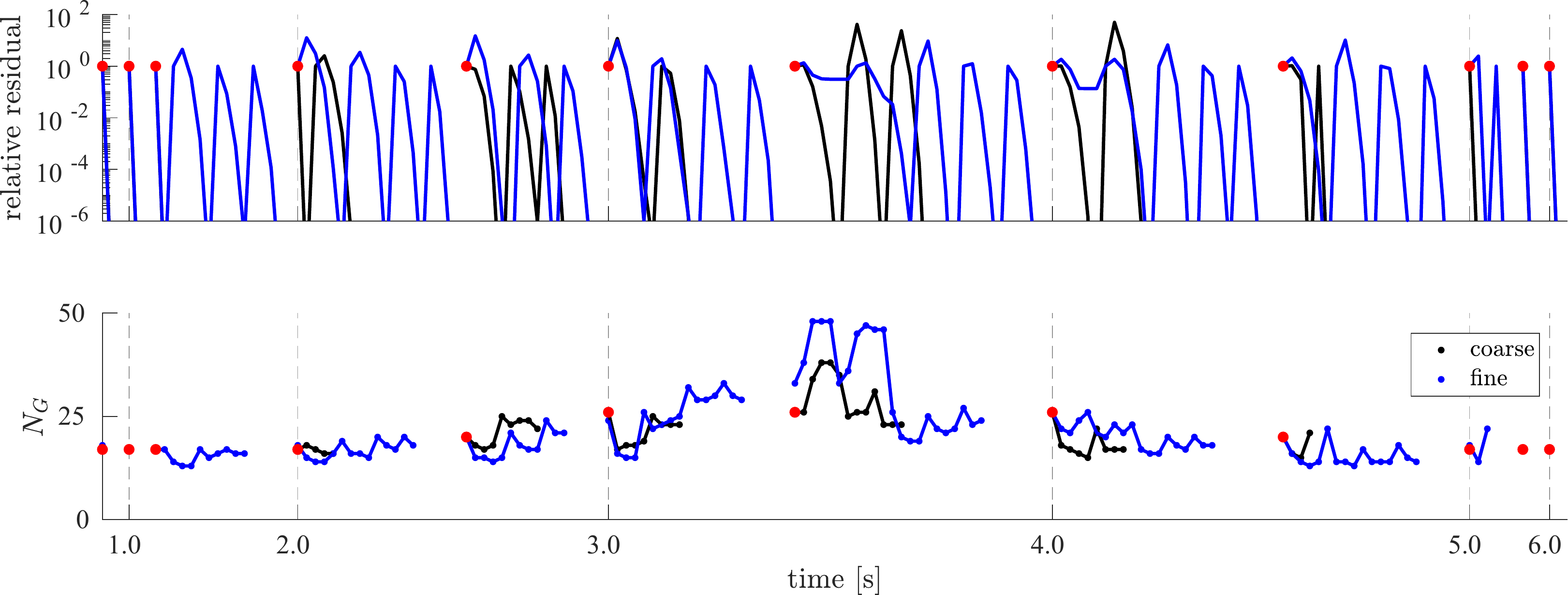}
  \caption{Test 2a: Non-linear convergence profiles (a) and number of linear iterations for each Jacobian system solution.}
  \label{fig:t2_convNonLin}
\end{figure}

The simulation is built so that all operating modes are experimented along the fracture surfaces. At the beginning, all elements are in stick conditions, then the slip and open regions progressively increase until $t=3$ s, where almost all fractures are open. Starting from $t=3.5$ s, the elements start to close again until they return to the initial condition at $t=6$ s. When a large number of fracture elements belongs to either the slip or the open region, the non-linear problem is more difficult and requires more active-set steps and inner Newton's iterations to converge. Quite intuitively, a larger number of fracture elements should require more non-linear iterations, as it can be appreciated in Table \ref{tab:t2_nnlin}. The coarse mesh totals $N_{\ell}=20$ and $N_N=75$, while the fine mesh requires $N_{\ell}=35$ and $N_N=154$.
By distinction, the linear solver performance is practically unaffected by the mesh refinement. The bottom panel of Figure \ref{fig:t2_convNonLin} and Table \ref{tab:t2_lin} show the number of GMRES iterations $N_G$ at every Jacobian system solution and the average value $\overline{N}_G$ for each time step, respectively. Notice that the linear iteration count tends to increase when the non-linear problem is more difficult, i.e., around $t=3$ s.
This is due to the different space and time refining indeed, while the space discretization is halved, the time discretization remains the same between the coarse and fine grids, and the coupling among the different physical processes involved in the simulation changes.
However, the proposed solution method appears to be fully scalable with respect to the grid size.

\begin{table}
  \centering
  \begin{tabular}{c|cccccccccccc}
   $t$ [s] &  0.5 &  1.0 &  1.5 &  2.0 &  2.5 &  3.0 &  3.5 &  4.0 &  4.5 &  5.0 &  5.5 &  6.0 \\
  \hline
   coarse  & 17.0 & 17.0 & 17.0 & 16.8 & 21.2 & 21.3 & 28.8 & 18.3 & 18.0 & 17.0 & 17.0 & 17.0 \\
   fine    & 18.0 & 17.0 & 15.5 & 16.9 & 18.2 & 25.1 & 32.5 & 20.4 & 15.4 & 18.0 & 17.0 & 17.0 \\
  \end{tabular}
  \caption{Test 2a: Average number of linear iterations $\overline{N}_G$ for the coarse and fine mesh.}
  \label{tab:t2_lin}
\end{table}

The investigated problem couples embedded 2D structures, i.e., the fractures, with the variables living in 3D domain. Therefore, a progressive grid refinement changes also the relative size of the blocks appearing in the Jacobian matrix $\mathcal{J}$, thus potentially modifying the overall problem conditioning. 
To analyze the behavior of the proposed algorithm with different 2D-to-3D ratios, i.e., the value of $(n_t+n_p)$ with respect to $n_u$, a second test case is introduced (Test 2b) consisting of a unitary cube with 7 vertical fractures (Figure \ref{fig:FP_MeshPressure}). The model has been regularly refined six times, with the mesh size varying from $h = 0.1 \ \text{m}$ to $h = 0.0167 \ \text{m}$. The number of unknowns for each refinement level, along with the percentage of 3D and 2D variables with respect to the total, is listed in Table \ref{tab:sizes}.
As with Test 2a, 
fluid injection and extraction is prescribed at the center of each fracture simulating the action of a horizontal well. The external faces parallel to the fractures are subjected to
a compressive constant load ($\sigma_0 = 10$ MPa), while the displacement on the other boundary faces is prevented. 
Figure \ref{fig:FP_MeshPressure} also shows the pressure solution at $t=3$ s.

\begin{table}
  \centering
  \begin{tabular}{c|c|c|c|c|c|c}
    level &   1 &      2 &      3 &      4 &      5 &      6 \\
    \hline
    cells & $10\times10\times10$ & $20\times20\times20$ &  $30\times30\times30$ &    $40\times40\times40$ &  $50\times50\times50$ & $60\times60\times60$ \\
    \hline
    $n_u$ &   4,668 &  31,050 &  97,176 & 221,046 & 420,600 & 714,018 \\
    $n_t$ &    972 &   3,888 &   8,749 &  15,552 &  24,300 &  34,992 \\
    $n_p$ &    324 &   1,296 &   2,916 &   5,184 &   8,100 &  11,664 \\
    \hline
    total &   5,964 &  36,234 & 108,840 & 241,782 & 453,060 & 760,674 \\
    \hline
    3D & 78.3\% & 85.7\% & 89.3\% & 91.4\% & 92.8\% & 93.9\% \\
    2D & 21.7\% & 14.3\% & 10.7\% &  8.6\% &  7.2\% &  6.1\% \\
  \end{tabular}
  \caption{Test 2b: Mesh size and percentage of 3D ($n_u$) and 2D ($n_t+n_p$) variables with respect to the total for different refinement levels.}
  \label{tab:sizes}
\end{table}

\begin{figure}
  \centering
  \null\hfill
  \subfloat[]{\includegraphics[height=0.3\linewidth]{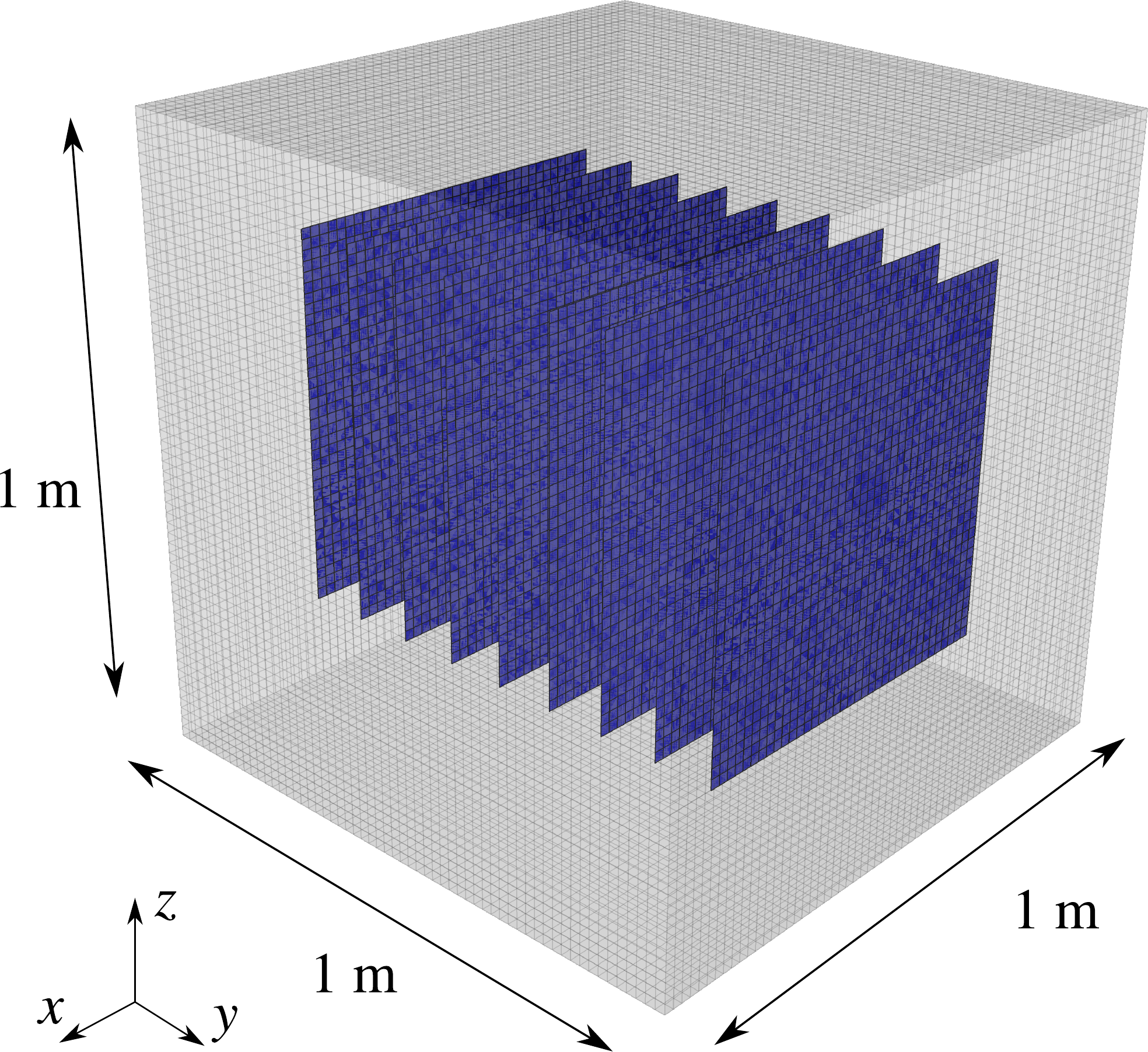}}\hfill
  \subfloat[]{\includegraphics[height=0.3\linewidth]{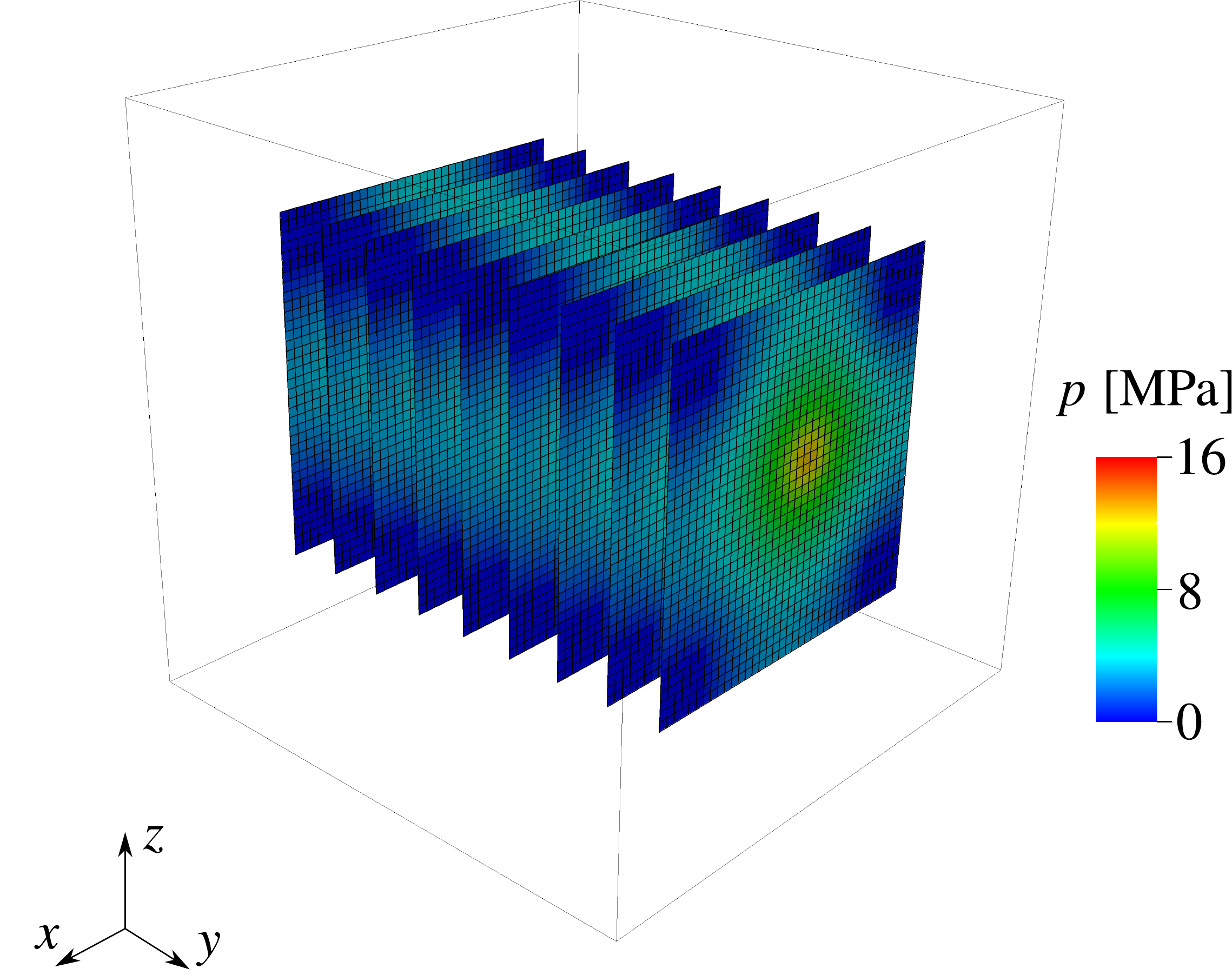}}
  \hfill\null
  \caption{Test 2b: Domain configuration (a) and pressure solution at time $t=3$ s (b).}
  \label{fig:FP_MeshPressure}
\end{figure}

Observe that after six refinements the size of the 2D blocks decreases from 21.7\% to 6.1\% only of the overall size of $\mathcal{J}$. Moreover, even with the last refinement level, which totals more than 750,000 unknowns, the size of $\tS_2$ is around 11,500, thus fully justifying the use of a nested direct solver. As already observed in Test 2a, it is expected that the overall number of non-linear iterations, i.e., $N_{\ell}$ and $N_N$, increases as the grid is progressively refined. 
This is observed in Figure \ref{fig:FP_nonLinIt}, which provides the relative variation of $N_{\ell}$ and $N_N$ with respect to the outcome obtained with the coarsest grid.
Figure \ref{fig:FP_meanLinIt} provides the average, maximum and
minimum number of GMRES iterations required by the linear solver. The average value is practically constant around 17 iterations, with the oscillations between the maximum and minimum iteration count comprised between 22 and 13. 
Hence, a very stable behavior of the proposed algorithm is obtained also changing the relative size of the matrix blocks in $\mathcal{J}$.

\begin{figure}
  \centering
  \null\hfill
  \subfloat[]{\includegraphics[width=0.4\linewidth]{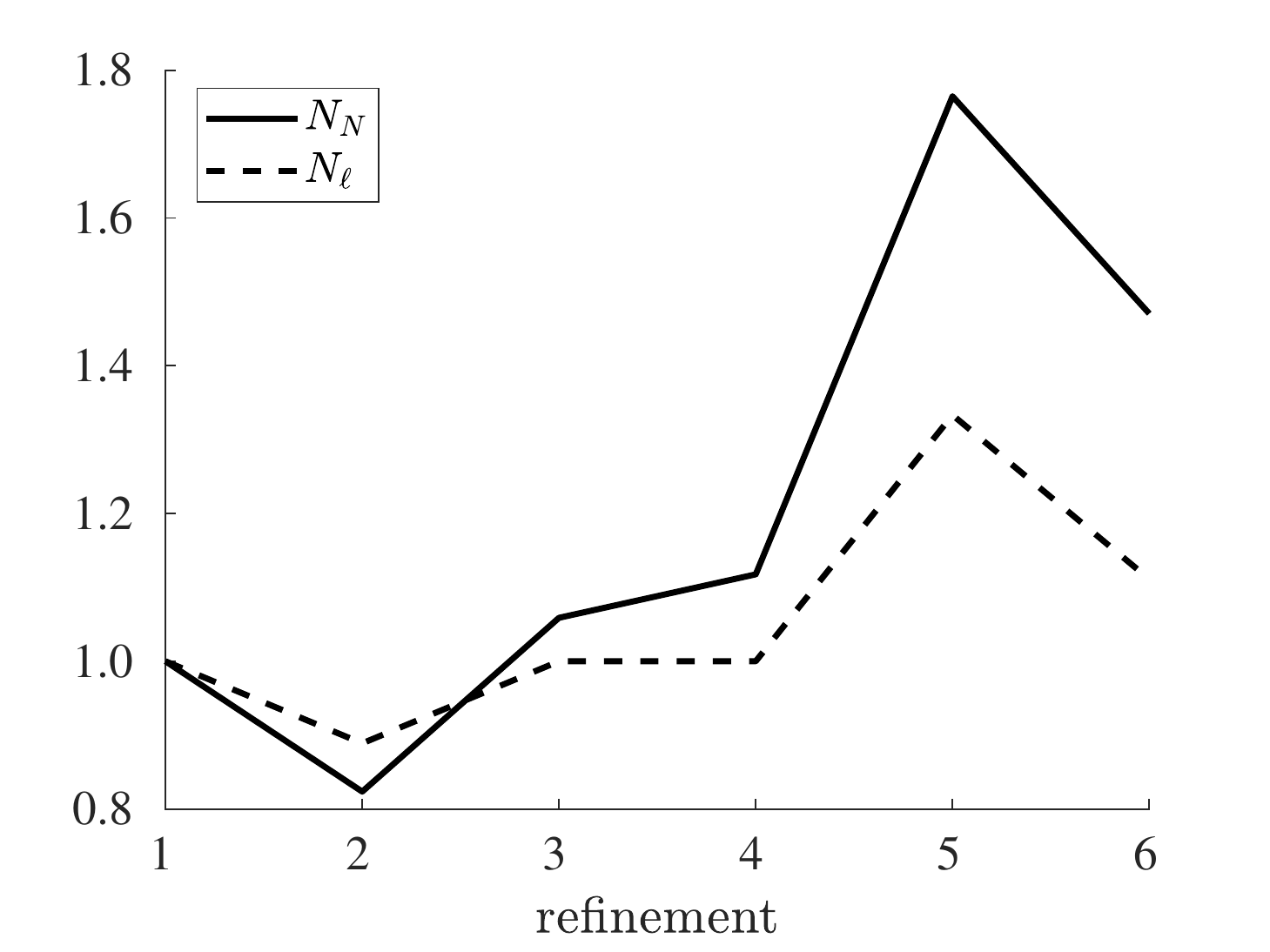}
    \label{fig:FP_nonLinIt}}\hfill
  \subfloat[]{\includegraphics[width=0.4\linewidth]{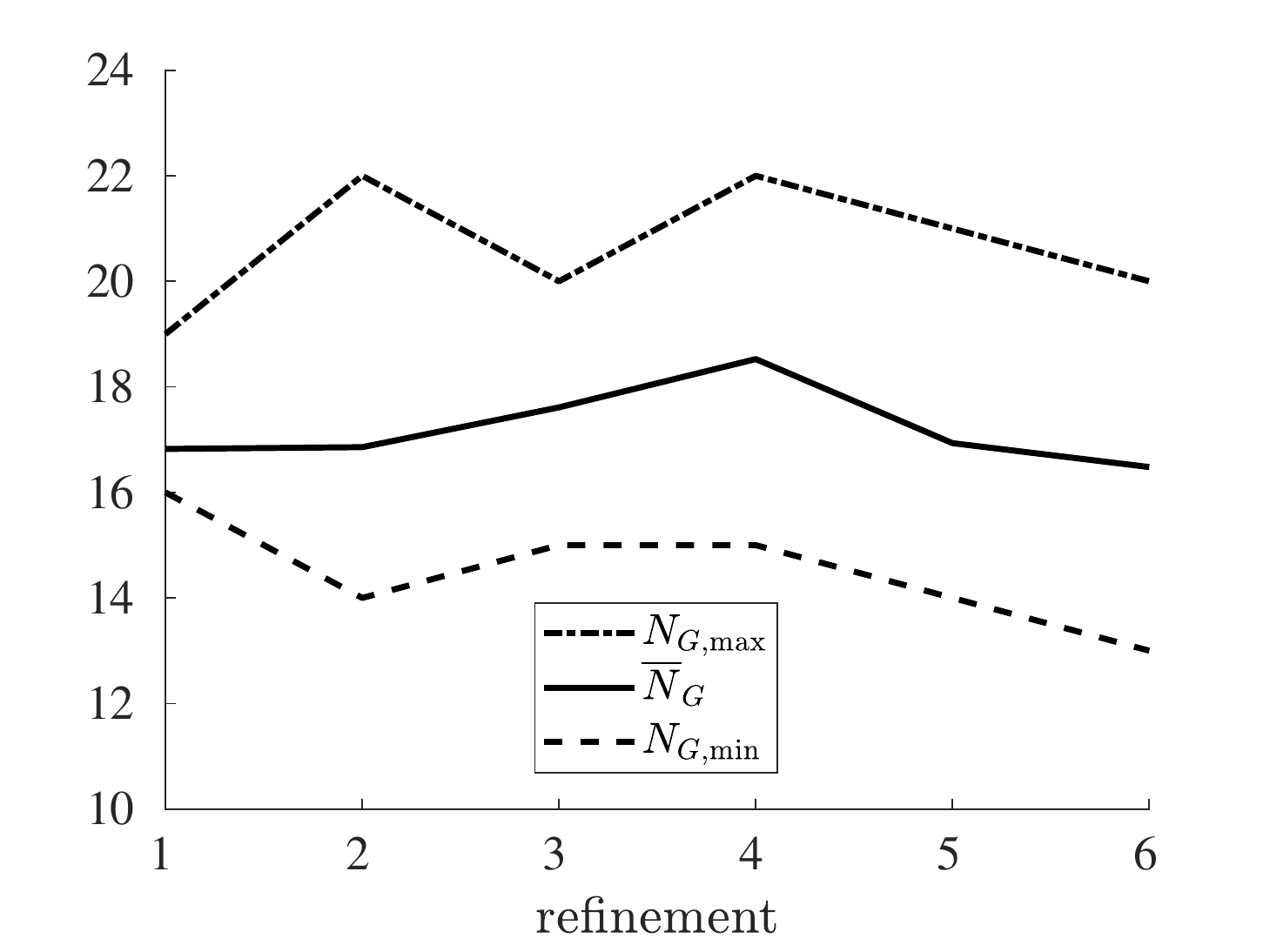}
    \label{fig:FP_meanLinIt}}
  \hfill\null
  \caption{Test 2b: active-set and Newton's iterations relative to those for the coarsest grid (a) and average, maximum and minimum GMRES iterations (b) for each refinement level.}
\end{figure}


\subsection{Test 3: Computational efficiency}

Finally, the performance of the proposed preconditioning framework is verified in a realistic application.
We consider a test case simulating a tilted well that intersects several fractures.
The problem reproduces the situation met in real-world applications of hydraulic fracturing stimulation. 
The well inclination is -15$^\circ$ with respect to the horizontal plane. The model analyzes the dynamics of 9 fractures located along the well in a $5.0\times1.6\times2.3$ m$^3$ box. The fractures
have the same size, but different relative positions with respect to the well, i.e., the well does
not intersect all of them at the same location.
The domain undergoes a compressive load parallel
to the fractures ($\sigma_0 = 10$ MPa), while displacements are prevented on the other boundary faces. The pressure at the four corners of every fractures is set to 0.
The model totals 342,642 nodes and 5,184 fracture
elements, corresponding to $n_u=1,027,926$, $n_t=15,552$, $n_p=5,184$, and an overall system size of 1,048,662 unknowns. 
The problem turns out to be particularly challenging because of the grid distortion and the different stick/slip/open region partitioning simultaneously obtained in each fracture. An example of the pressure solution obtained at $t=3.5$ s is reported in Figure \ref{fig:RW_press}.

\begin{figure}
  \centering
  \null\hfill
  \subfloat[]{\includegraphics[height=0.25\linewidth]{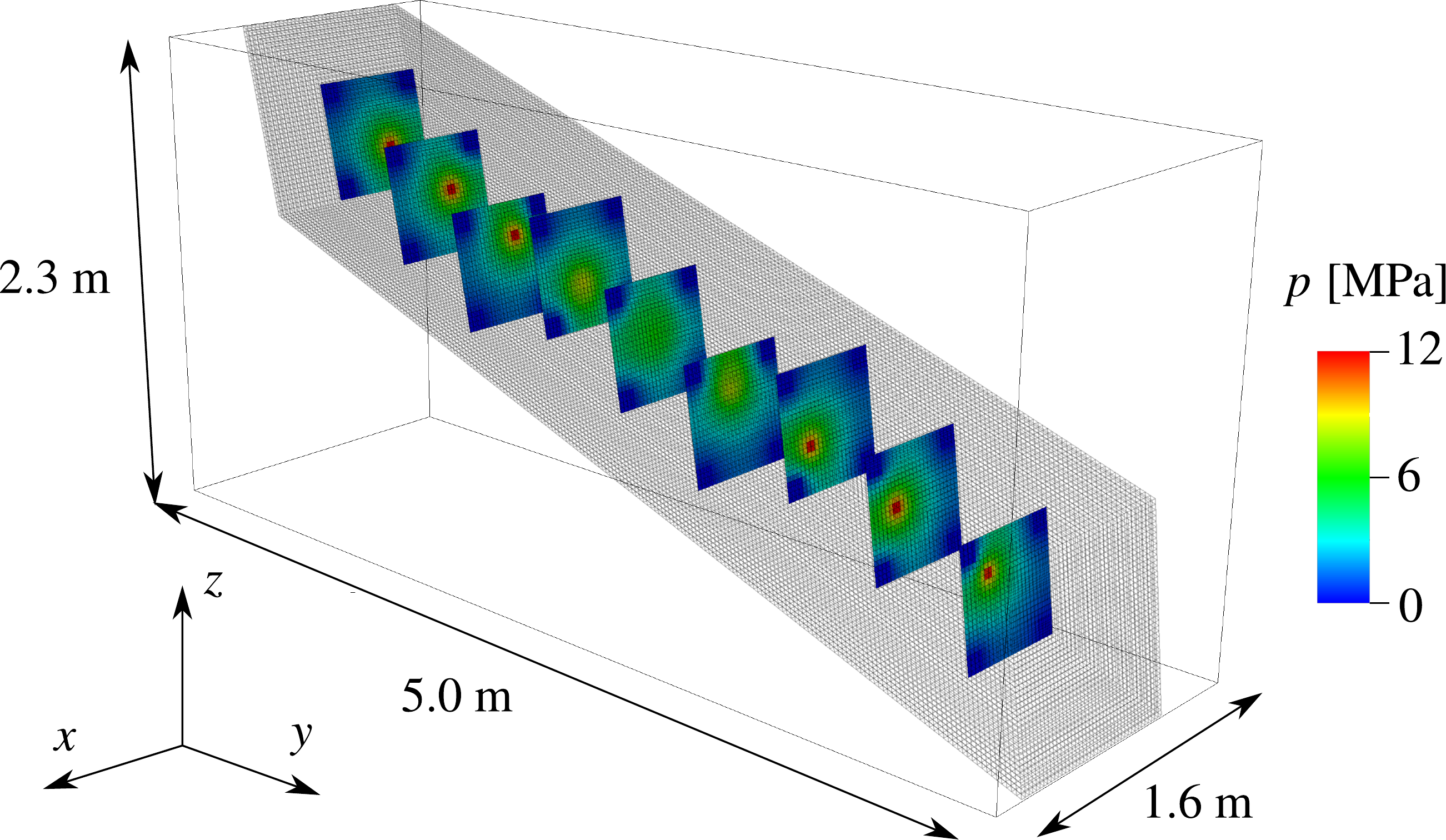}
    \label{fig:RW_press}}\hfill
  \subfloat[]{\includegraphics[height=0.25\linewidth]{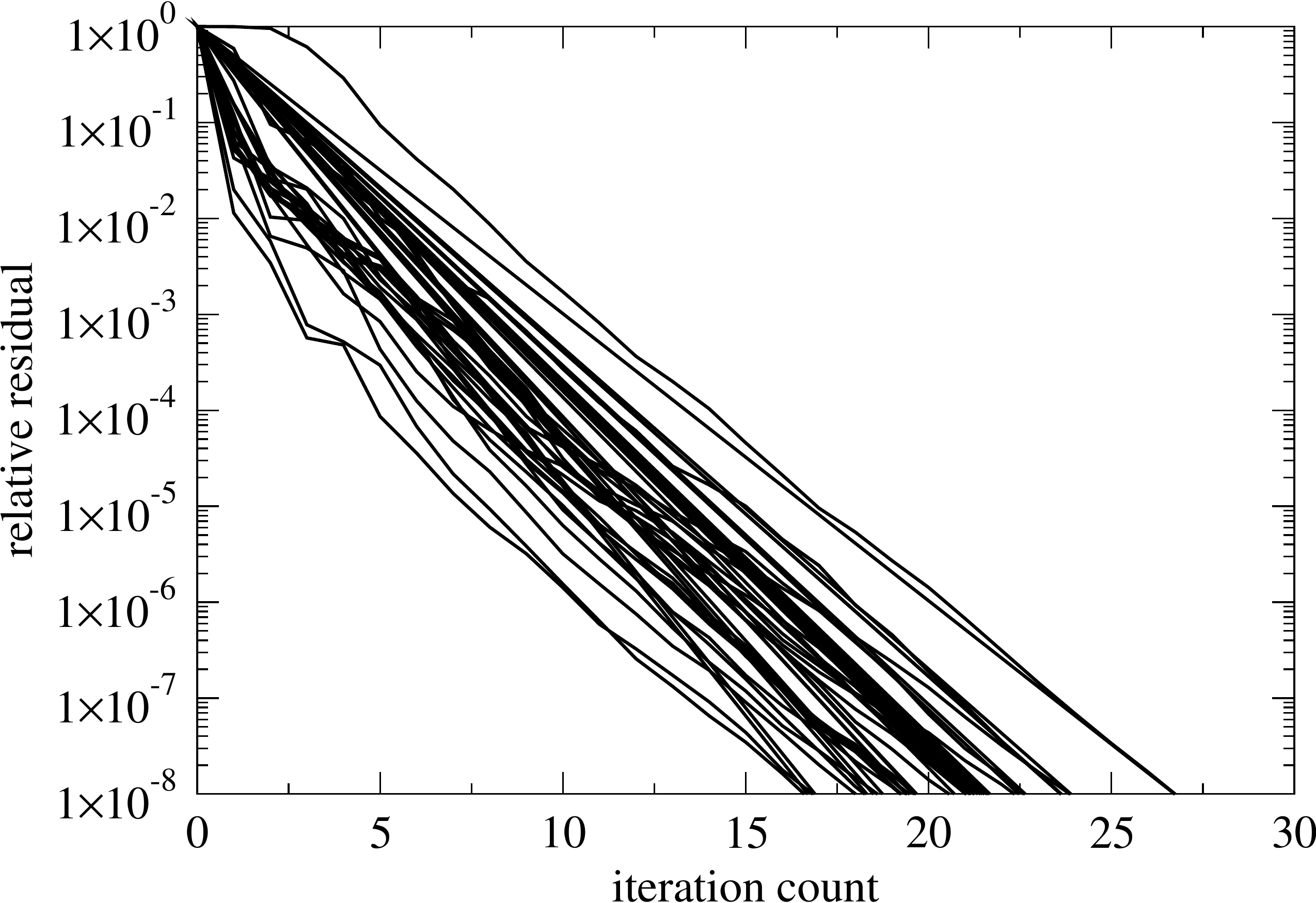}
    \label{fig:RW_conv}}
  \hfill\null
  \caption{Test 3: Pressure solution at $t=3.5$ s (a) and GMRES convergence profiles for all the Newton iterations at the same time (b).}
  \label{fig:RW_convPress}
\end{figure}

The overall non-linear simulation is very demanding, with several active-set steps and restarts for the inner Newton's loop. The performance of the non-linear solution algorithm is summarized in Table \ref{tab:RW_nolin}, which provides the active-set steps and Newton's loop iterations at every simulation time. The irregular and non-uniform behavior of the different fractures causes convergence difficulties at almost every time-step, with an overall number of active-set steps and Newton's iterations equal to $N_{\ell}=49$ and $N_N=238$. The full simulation, run on a single node of a small cluster equipped with 16 Intel(R) Xeon(R) Gold 6130 @ 2.10GHz CPU and 6 TB of RAM, requires a total CPU time of $6472.5$ s.

\begin{table}
  \centering
  \begin{tabular}{c|c|c|c|c|c|c|c|c|c|c|c|c|}
    & \multicolumn{12}{|c|}{time [s]} \\
    \hline
    $\ell$ & 0.5 & 1.0 & 1.5 & 2.0 & 2.5 & 3.0 & 3.5 & 4.0 & 4.5 & 5.0 & 5.5 & 6.0 \\
    \hline
    1 & 2 & 2 & 6 & 8 & 8 & 16 & 2 & 5 & 5 & 4 & 3 & 2 \\
    2 &   & 5 & 6 & 6 & 6 &  6 & 6 & 6 & 6 & 4 & 2 &   \\
    3 &   & 4 & 5 & 5 & 5 &  6 & 5 & 5 & 4 & 2 &   &   \\
    4 &   & 5 & 5 & 5 & 5 &  5 & 5 & 5 & 4 & 3 &   &   \\
    5 &   &   & 4 & 5 & 5 &  4 & 4 & 4 & 5 &   &   &   \\
    6 &   &   &   &   &   &    &   & 4 & 4 &   &   &   \\
    \hline
    $\overline{N}_G$ & 25.0 & 22.3 & 23.3 & 25.9 & 29.2 & 26.8 & 22.9 & 21.2 & 20.5 & 22.1 & 21.3 & 21.0 \\
  \end{tabular}
  \caption{Test 3: Newton's iterations for each simulation time and active-set step $\ell$. The average number of GMRES iterations $\overline{N}_G$ per time step is also reported. 
  }
  \label{tab:RW_nolin}
\end{table}

Despite the challenges posed by the overall simulation, the performance of the linear solver proves very stable and efficient.
The average number of GMRES iterations is reported in Table \ref{tab:RW_nolin} for each simulation time. Over the entire simulation, we have $\overline{N}_G=24.2$, with $N_{G,\min}=16$ and $N_{G,\max}=185$. The latter corresponds to a Newton step where a large number of elements move from one region to another. As an example, Figure \ref{fig:RW_conv} shows all the convergence profiles obtained by a right-preconditioned GMRES accelerated by the t-u-p approach for the Newton loop at $t = 3.5s$. It can be noticed the great stability of the solver behavior, even if the fracture state changes significantly. 
The average CPU time required in this simulation for a single system solution is $33.4$ s.


\section{Conclusions}
\label{sec:concl}

The simulation of frictional contact mechanics with fluid flow in the fracture network is an important problem in several engineering applications.
The mathematical model can be numerically solved with the aid of a blended finite element/finite volume formulation, giving rise to a strongly non-linear problem addressed by an active-set strategy coupled with an inner Newton iteration.
At each Newton step, a linear system with a $3\times3$ block Jacobian matrix has to be solved.
Since standard global approaches cannot be effectively used with the resulting non-symmetric and indefinite matrix,
this work focused on the development of a robust, scalable and efficient preconditioning framework for the solution of the inner linear problem.


The algebraic properties of the $3\times3$ block Jacobian matrix change during a full simulation with the evolution of the stick/slip/open region partitioning of the fractures. In particular, different couplings may arise and disappear, with the Schur complements possibly being symmetric positive definite, only positive definite or non-symmetric and indefinite.
The proposed preconditioning framework exploits the physics-based variable partitioning and the use of multigrid techniques for the sake of algorithmic scalability.
The basic idea relies on restricting the system to a single-physics problem, approximately solve it by an inner AMG, and prolong the solution back to the full multi-physics problem.
In particular, two multi-physics reduction sequences are developed, denoted as t-p-u (traction-pressure-displacement) and t-u-p (traction-displacement-pressure) approaches, and compared in a set of numerical examples. 
The results that follow are worth summarizing.
\begin{itemize}
    \item Theoretical analyses show that the proposed approaches are expected to have a similar convergence rate, with a slight advantage for the t-p-u approach because of a more clustered eigenvalue distribution for the preconditioned matrix and a smaller application cost. Indeed, this is confirmed by the numerical experiments if nested direct solvers are used in the preconditioner application, but this approach may soon lose robustness when AMG methods are introduced. The reason stems from the possible indefiniteness of the arising Schur complement, which is avoided in the t-u-p approach.
    \item The proposed approach proves to be algorithmically scalable with respect to the computational grid size and the relative size of the discrete fracture network to the full 3D domain. Although the non-linear problem can become harder to solve, the iteration count for the inner linear solver is independent on the discretization size.
    \item The application in a realistic configuration, simulating a hydraulic fracturing stimulation through a tilted well, shows the computational efficiency of the proposed approach. Despite the difficulty met by the non-linear algorithm, due to the combination of variable stick/slip/open operating modes for the different fractures, the linear solver exhibits a very stable behavior throughout the full simulation and a remarkable efficiency also in a sequential implementation.
\end{itemize}
Future developments regard the implementation of the proposed preconditioning framework in high performance computing
infrastructures, in order to fully exploit the algorithmic scalability and test the actual parallel efficiency, and the extension to (multi-phase) fluid flow in the porous matrix as well.

\section*{Acknowledgements}
\label{sec::acknow}
Partial funding was provided by TotalEnergies through the FC-MAELSTROM project.
Portions of this work were performed 
within the 2020 INdAM-GNCS project ``Optimization and advanced linear algebra for PDE-governed problems''.
Computational resources were provided by University of Padova Strategic Research
Infrastructure Grant 2017: ``CAPRI: Calcolo ad Alte Prestazioni per la Ricerca e
l’Innovazione''.



\end{document}